\documentclass[a4paper,11pt,reqno]{article}

\usepackage[utf8]{inputenc}
\usepackage[T1]{fontenc}
\usepackage{lmodern}
\usepackage[english]{babel}
\usepackage{microtype}

\usepackage{amsmath,amssymb,amsfonts,amsthm}
\usepackage{mathtools}
\usepackage{aliascnt}
\usepackage{mathrsfs}
\usepackage{braket}
\usepackage{bm}

\usepackage[a4paper,hmargin=3.25cm,bottom=4cm,top=3.5cm,footskip=3\baselineskip]{geometry}
\usepackage[colorlinks,citecolor=blue]{hyperref}

\usepackage{tikz}

\usepackage{enumitem}
\setlist[enumerate,1]{label=(\roman*), ref=\roman*}
\setlist[enumerate,2]{label=\alph*., ref=\theenumi.\alph*}
\setlist[enumerate,3]{label=\arabic*., ref=\theenumii.\arabic*}

\AtBeginDocument{}

\makeatletter
\def\newaliasedtheorem#1[#2]#3{
  \newaliascnt{#1@alt}{#2}
  \newtheorem{#1}[#1@alt]{#3}
  \expandafter\newcommand\csname #1@altname\endcsname{#3}
}
\makeatother

\numberwithin{equation}{section}

\theoremstyle{plain}
\newtheorem{theorem}{Theorem}[section]
\newaliasedtheorem{proposition}[theorem]{Proposition}
\newaliasedtheorem{lemma}[theorem]{Lemma}
\newaliasedtheorem{corollary}[theorem]{Corollary}
\newaliasedtheorem{counterexample}[theorem]{Counterexample}

\theoremstyle{definition}
\newaliasedtheorem{definition}[theorem]{Definition}
\newaliasedtheorem{question}[theorem]{Question}
\newaliasedtheorem{example}[theorem]{Example}

\theoremstyle{remark}
\newaliasedtheorem{remark}[theorem]{Remark}

\newcommand{\setN}{\mathbb{N}}

\newcommand{\setR}{\mathbb{R}}

\newcommand{\eps}{\varepsilon}

\DeclarePairedDelimiter{\abs}{\lvert}{\rvert}
\DeclarePairedDelimiter{\norm}{\lVert}{\rVert}

\DeclarePairedDelimiter{\moment}{[\hspace{-1.5pt}[}{]\hspace{-1.5pt}]}

\let\d\undefined
\newcommand{\d}{\mathop{}\!\mathrm{d}}
\newcommand{\D}{\mathop{}\!\mathrm{D}}

\DeclareMathOperator{\Id}{Id}

\DeclareMathOperator{\diam}{diam}

\newcommand{\haus}{\mathscr{H}}
\newcommand{\leb}{\mathscr{L}}

\newcommand{\Prob}{\mathscr{P}}

\newcommand{\res}{\mathop{\raisebox{-.127ex}{\reflectbox{\rotatebox[origin=br]{-90}{$\lnot$}}}}}
\DeclareMathOperator{\Exp}{\mathbb{E}}
\DeclareMathOperator{\Pro}{\mathbb{P}}
\DeclareMathOperator{\Var}{Var}

\newcommand{\meas}{\mathfrak{m}}

\DeclareMathOperator{\Lip}{Lip}

\newcommand{\plchldr}{\,\cdot\,}

\newcommand{\dist}{\mathsf{d}}
\newcommand{\T}{\mathbb{T}}
\newcommand{\Beta}{\mathrm{B}}

\newcommand{\Cuc}{C_{\mathrm{uc}}}
\newcommand{\Csg}{C_{\mathrm{sg}}}
\newcommand{\Cdr}{C_{\mathrm{dr}}}
\newcommand{\Cg}{C_{\mathrm{ge}}}
\newcommand{\Cr}{C_{\mathrm{rt}}}

\newcommand\blfootnote[1]{%
  \begingroup
  \renewcommand\thefootnote{}\footnote{#1}%
  \addtocounter{footnote}{-1}%
  \endgroup
}

\title{A PDE approach to a 2-dimensional matching problem}

\author{
Luigi Ambrosio
\thanks{Scuola Normale Superiore, \url{luigi.ambrosio@sns.it}.} \and
Federico Stra
\thanks{Scuola Normale Superiore, \url{federico.stra@sns.it}.} \and
Dario Trevisan
\thanks{Università di Pisa, \url{dario.trevisan@unipi.it}. Partially supported by Project PRA\_2016\_41.}}

\begin{document}

\maketitle

\begin{abstract}
We prove asymptotic results for 2-dimensional random matching problems. In particular, we obtain the leading term in the asymptotic expansion of the expected quadratic transportation cost for empirical measures of two  samples of independent uniform random variables in the square. Our technique is based on a rigorous formulation of the challenging PDE ansatz by S.\ Caracciolo et al.\ (Phys. Rev. E, {\bf 90} 012118, 2014) that ``linearise'' the Monge-Ampère equation.
\blfootnote{\textit{MSC-2010:} primary 60D05; secondary 49J55, 60H15.}
\end{abstract}

\tableofcontents

\section{Introduction}

Optimal matching problems are random variational problems widely investigated in the mathematical and physical literature. 
Many variants are possible, for instance the monopartite
problem, dealing with the optimal coupling of an even number $n$ of i.i.d. points $X_i$,  
the grid matching problem, where one looks for the optimal matching of
an empirical measure $\sum_i\frac 1n\delta_{X_i}$ to a deterministic and ``equally spaced'' grid, 
the closely related problem of optimal matching to the common law $\meas$ of $X_i$ and the bipartite
problem, dealing with the optimal matching of $\sum_i\frac 1n\delta_{X_i}$ to $\sum_i\frac 1n\delta_{X_i}$,
with $(X_i,Y_i)$ i.i.d. See the monographs \cite{Y98} and \cite{T14} for many more informations on this subject.
In addition to these problems, one may study
the optimal assignment problem, \cite{C04}, where the optimization involves also the weights of the
Dirac masses $\delta_{X_i}$  and the closely related problem of transporting Lebesgue measure to
a Poisson point process \cite{HS13}, which involves in the limit measures with infinite mass.

In this paper we focus on two of these problems, namely optimal matching to the reference measure and the
bipartite problem. Denoting by $D$ the $d$-dimensional domain and by $\meas\in\Prob(D)$ the law of the points
$X_i$, $Y_i$, the problem is to estimate the rate of convergence to $0$ of
\begin{equation}\label{eq:ambr30}
\Exp\left[W_p^p\left(\sum_{i=1}^n\frac 1n\delta_{X_i},\meas\right)\right],\qquad
\Exp\left[W_p^p\left(\sum_{i=1}^n\frac 1n\delta_{X_i},
	\sum_{i=1}^n \frac 1n\delta_{Y_i}\right)\right],
\end{equation}
where $p\in [1,\infty)$ is the power occurring in the transportation cost $c=\dist^p$
(also the case $p=\infty$ is considered in the literature, see for instance \cite{SY91} and the
references therein), finding tight upper and lower bounds and, possibly, proving existence of the
limit of the renormalized quantities as $n\to\infty$.

The typical distance between points is expected to be of order $n^{-1/d}$, and therefore it is natural to guess that the quantities $c_{n,p,d}$ introduced in \eqref{eq:ambr30} behave as $n^{-p/d}$. However, it is by now well
known that this hypothesis is true for $d\geq 3$, while it is false for $d=1$ and $d=2$. Despite
plenty of heuristic arguments and numerical results, these are (as far as we know) the main results that have been rigorously 
proved (we focus here on the model case when $\meas$ is the uniform measure and we do not distinguish between optimal matching to $\meas$ and bipartite), denoting $a_n\sim b_n$ if $\limsup_n a_n/b_n<\infty$ and $\limsup_n b_n/a_n<\infty$:
\begin{itemize}
\item when $D=[0,1]$ or $D=\T^1$, then $c_{n,p,1}/n^{-p}\sim n^{p/2}$ and, when $p=2$, $\lim\limits_{n\to\infty} n c_{n,2,1}$ can be explicitly computed, see \cite{CS14};
\item when $D=[0,1]^2$, then $c_{n,p,2}/n^{-p/2}\sim (\log n)^{p/2}$, see \cite{AKT84};
\item when $D=[0,1]^d$ with $d\geq 3$, then $c_{n,1,d}/n^{-1/d}\sim 1$ and the limit exists \cite{BM02}, \cite{DY95}, for general $p>1$ and $2p>d$ one has $c_{n,p,d}/n^{-p/d}\sim 1$ and the limit exists \cite{B13}; a combination of these results and H\"older's inequality gives $c_{n,p,d}/n^{-p/d}\sim 1$ for $p\in[1,\infty)$ and $d\geq3$, but it is not known whether the limit exists for $p\in(1,d/2)$.  In the more recent paper \cite{FG15} also non-asymptotic upper bounds have been provided.
\end{itemize}
Notice that some of the results listed above provide not only convergence of the expectations, but also almost sure
convergence which, under some circumstances (see for instance \cite{B13}) can be obtained from concentration inequalities
as soon as convergence of the expectations is known. In the case $d=2$, the convergence of $(\log n)^{-p/2}c_{n,p,2}/n^{-p/2}$ as $n\to\infty$ and the characterization of the limit are still open problems, particularly in the case $p=1$ \cite[Research problem 4.3.3]{T14}.

Our interest in this subject has been motivated by the recent work \cite{CLPS14} where, on the basis
of an ansatz, very specific predictions on the expansion of 
\[
n^{-p/d}\Exp\left[W_p^p\left(\sum_{i=1}^n\frac 1n\delta_{X_i},
	\sum_{i=1}^n\frac 1n\delta_{Y_i}\right)\right]
\]
have been made on $\T^d$, for all ranges of dimensions $d$ and powers $p$. In brief, the ansatz of \cite{CLPS14} 
is based on a linearisation ($\rho_i\sim 1$ in $C^1$ topology, $\psi\sim f+\frac 12 \abs{x}^2$ in $C^2$ topology)
of the Monge-Amp\`ere equation
\[
\rho_1(\nabla\psi){\rm det}\nabla^2\psi=\rho_0,
\]
(which describes the optimal transport map $T=\nabla\psi$ from the measures having probability densities $\rho_0$ to $\rho_1$) leading to Poisson's equation $-\Delta f=\rho_1-\rho_0$.

This ansatz is very appealing, but on the mathematical side it poses several challenges, because the energies involved are infinite for $d\geq 2$ (the measures being Dirac masses), because this procedure does not provide an exact matching between the measures (due to the linearisation) and because the necessity of giving lower bounds persists, as matchings provide only upper bounds. While we are still very far from justifying rigorously all predictions of \cite{CLPS14}, see also \autoref{sec:more_general_spaces} for a discussion on this topic, we have been able to use this idea to prove existence of the limit and compute explicitly it in the case $p=d=2$, in agreement with \cite{CLPS14}: 

\begin{theorem}[Main result]\label{thm:main} Assume that either $D=[0,1]^2$ or that $D$ is a compact 2-dimensional
Riemannian manifold with no boundary and unit volume $\meas(D)=1$ and set $\mu^n=\sum_i\frac 1n\delta_{X_i}$,
$\nu^n=\sum_i\frac 1n\delta_{Y_i}$, being $X_i$, $Y_i$\ i.i.d.\ with law $\meas$. Then,
\[
\lim_{n\to\infty} \frac{n}{\log n} \Exp\left[W_2^2(\mu^n,\meas)\right]
= \frac 1{4\pi}.
\]
In the bipartite case, if either $D=[0,1]^2$ or $D=\T^2$, one has
\begin{equation}\label{eq:ambr25}
\lim_{n\to\infty} \frac{n}{\log n} \Exp\left[W_2^2(\mu^n,\nu^n)\right]
= \frac 1{2\pi}.
\end{equation}
Finally, in the case $D=[0,1]^2$, if $T^{\mu^n}$ denotes the optimal transport map from $\meas$
to $\mu^n$, one has
\begin{equation}\label{eq:ambr27}
\lim_{n\to\infty}\frac{n}{\log n} \int_D \abs*{\Exp\bigl[T^{\mu^n}(x)-x\bigr]}^2 \d\meas(x) = 0.
\end{equation}
\end{theorem}

In our proof the geometry of the domain $D$ enters only through the (asymptotic) properties of the spectrum of
the Laplacian with Neumann boundary conditions; for this reason we are able to cover also abstract manifolds (where another example of interest could be the two dimensional sphere).
 Even though in dimension $d=1$ (but mostly for the case $D=[0,1]$) a much more detailed analysis
can be made, see for instance \autoref{rem:1deasy}, we include proofs and statements of the 1-$d$ case, to illustrate the flexibility of our
synthetic method.

Let us give some heuristic ideas on the strategy of proof, starting from the upper bound.
In order to obtain finite energy solutions to Poisson's equation
we study the regularized PDE
\begin{equation}\label{eq:Poisson_intro}
-\Delta f^{n,t}=(u^{n,t}-1)
\end{equation}
where $u^{n,t}$ is the density of $P_t^*(\mu^n-\meas)$ and $P^*_t$ is the heat semigroup with Neumann boundary conditions,
acting on measures. Then, choosing $t=\gamma n^{-1}\log n$
with $\gamma$ small, we have a small error in the estimation from above of $c_{n,2,2}$ if we replace $\mu^n$ by
its regularization $P_t^*\mu^n$. Eventually, we use Dacorogna-Moser's technique (see \autoref{prop:dacorogna-moser}) 
to provide an exact coupling between $P_t^*\mu^n$ and $\meas$, leading to an estimate of the form
\[
W_2^2(P_t^*\mu^n,\meas) \leq
\int_D\biggl(\int_0^1\frac 1{(1-s)+su^{n,t}}\d s\biggr)
	\abs{\nabla f^{n,t}}^2\d\meas.
\]
To conclude, we have to estimate very carefully how much the factor in front of $\abs{\nabla f^{n,t}}^2$ differs from 1; this requires
in particular higher integrability estimates on $\abs{\nabla f^{n,t}}$.

Let us consider now the lower bound. The duality formula
\[
\frac 12 W_2^2(\mu,\nu)=\sup_{\phi(x)+\psi(y)\leq \dist^2(x,y)/2}-\int_D\phi\d\mu+\int_D\psi\d\nu
\]
is the standard way to provide lower bounds on $W_2$; given $\phi$, the best possible $\psi=Q_1\phi$ compatible with the constraint is given by the Hopf-Lax formula \eqref{eq:ambr6}. Choosing again $\phi=f^{n,t}$ as in the ansatz, we are led to estimate carefully 
\[
\frac 12\int \abs{\nabla f^{n,t}}^2-\biggl(-\int_D f^{n,t}\d\meas+\int_DQ_1 f^{n,t}\d\meas\biggr)
\]
in events of the form $\{\abs{u^{n,t}-1}\leq\eta\}$ (whose probabilities tend to 1). We do this using Laplacian estimates and the viscosity approximation of the Hopf-Lax semigroup provided by the Hopf-Cole transform. 

In the bipartite case, the result can be obtained from the previous ones playing with independence. Heuristically, the random ``vectors'' pointing from $\meas$ to $\mu^n$ and from $\meas$ to $\nu^n$ are independent, and since $\Prob(D)$ is ``Riemannian'' on small scales when endowed with the distance $W_2$, we obtain a factor $2$, as in the identity $\Exp[ (X-Y)^2] = 2\Var(X)$ when $X$, $Y$ are i.i.d.\ random variables. Interestingly, the rigorous proof of this fact provides also the information \eqref{eq:ambr27} on the mean displacement as function of the position.

The paper is organized as follows. In \autoref{sec:prel} we first recall preliminary results on the Wasserstein distance and the main tools (Dacorogna-Moser interpolation, duality, Hopf-Lax semigroup) involved in the proof of the upper and lower bounds. Then, we provide moment estimates for $\sqrt{n}(\mu^n-\meas)$.

In \autoref{sec:heat} we introduce the heat semigroup $P_t$ and, in a quantitative way, the regularity properties of $P_t$ needed for our scheme to work. We also provide estimates on the canonical regularization of the Hamilton-Jacobi equation provided by the Hopf-Cole transform $-\sigma (\log P_t e^{-f/\sigma})$. The most delicate part of our proof involves bounds on the probability of the events
\[
\left\{\sup_{x\in D}\abs{u^{n,t}(x)-1}>\eta\right\}, \qquad \eta>0
\]
which ensure that the probability of these events has a power like decay as $n\to\infty$ if $t=\gamma n^{-1}\log n$, with
$\gamma$ sufficiently large (this plays a role in the proof of the lower bound). Finally, in light of the ansatz of \cite{CLPS14}, we provide a formula for
\[
\Exp\left[\int_D \abs{\nabla f^{n,t}}^2\d\meas\right],
\]
where $f^{n,t}$ solves the random PDE \eqref{eq:Poisson_intro},  and prove convergence of the renormalized quantity as $n\to\infty$, if $t\sim n^{-1}\log n$.

\autoref{sec:proofmain} provides the proof of our main result, together with \autoref{thm:main1} dealing with the simpler case $d=1$. We first deal with the optimal matching to $\meas$, and then we deal with the bipartite case.

In \autoref{sec:newAKL} we recover the result found in \cite{AKT84} as a consequence of our estimates via a Lipschitz approximation argument.

Finally, \autoref{sec:more_general_spaces} covers extensions to more general classes of domains and open problems, pointing out some potential developments.

\paragraph{Acknowledgment.} The first author warmly thanks S.\ Caracciolo for pointing out to him the paper \cite{CLPS14} and for several conversations on the subject.

\section{Notation and preliminary results}\label{sec:prel}

\subsection{Wasserstein distance}

Let $(D,\dist)$ be a complete and separable metric space. We recall (see e.g.\ \cite{AGS08}) that the quadratic Wasserstein distance $W_2(\mu,\nu)$ between Borel probability measures $\mu$, $\nu$ in $D$ with finite quadratic moments is defined by
\[
W_2^2(\mu,\nu) = \min\left\{
	\int_{D\times D}\dist(x,y)^2\d\Sigma(x,y) :
	\Sigma\in\Gamma(\mu,\nu) \right\},
\]
where $\Gamma(\mu,\nu)$ is the class of transport plans (couplings in Probability) between $\mu$ and $\nu$, namely Borel probability measures $\Sigma$ in $D\times D$ having $\mu$ and $\nu$ as first and second marginals, respectively. We say that
a Borel map $T$ pushing $\mu$ to $\nu$ is optimal if 
\[
W_2^2(\mu,\nu)=\int_D \dist^2(T(x),x)\d\mu(x).
\]
This means that the plan $\Sigma=(\Id\times T)_\#\mu$ induced by $T$ is optimal.

The following
duality formula  will play a key role, both in the proof of the upper and lower bound of the matching cost:
\begin{equation}\label{eq:ambr5}
\frac 12 W_2^2(\mu,\nu)=\sup_{\phi\in\Lip_b(D)}-\int_D \phi\d\mu+\int_D Q_1\phi\d\nu. 
\end{equation} 
In \eqref{eq:ambr5} above, $\Lip_b(D)$ stands for the class of bounded Lipschitz functions on 
$D$ and, for $t>0$, $Q_t\phi$ is provided by the Hopf-Lax formula
\begin{equation}\label{eq:ambr6}
Q_t\phi(y)=\inf_{x\in D}\phi(x) + \frac1{2t}\dist(x,y)^2.
\end{equation}
This formula also provides a semigroup if $(X,\dist)$ is a length space, and $Q_t\phi\uparrow\phi$ as $t\downarrow 0$.

We recall a few basic properties of $Q_t$, whose proof is elementary: if $\phi\in\Lip_b(D)$ then
$\inf\phi\leq Q_t\phi\leq\sup\phi$ and (where $\Lip$ stands for the Lipschitz constant)
\[
\Lip(Q_t\phi)\leq 2\Lip(\phi), \qquad
\norm*{\frac{\d}{\d t}Q_t\phi(x)}_\infty \leq
2\bigl[\Lip(\phi)\bigr]^2 \quad\text{for all $x\in D$}.
\]
In particular $\Lip_b(D)$ is invariant under the action of $Q_t$. For $\phi\in \Lip_b(D)$, 
the key property of $Q_t\phi$ is  
\begin{equation}\label{eq:subsHj}
\frac{\d}{\d t}Q_t\phi+\frac 12\abs{\nabla Q_t\phi}^2\leq 0
\quad\text{$\meas$-a.e.\ in $X$, for all $t>0$},
\end{equation}
with equality if $(D,\dist)$ is a length space (but we will only need the inequality). In \eqref{eq:subsHj}, $\abs{\nabla Q_t\phi}$ is the metric slope of $Q_t\phi$, which corresponds to the norm of the gradient in the Riemannian setting.

We recall that $W_2^2$ is jointly convex, namely if $\mu_i,\nu_i\in\Prob(D)$, $t_i\geq 0$, $\sum_{i=1}^k t_i=1$,
\[
\mu = \sum_{i=1}^k t_i\mu_i, \qquad \nu = \sum_{i=1}^k t_i\nu_i
\]
then
\begin{equation}\label{eq:joint_convexity}
W_2^2(\mu,\nu) \leq \sum_{i=1}^k t_i W_2^2(\mu_i,\nu_i).
\end{equation}
This easily follows by the linear dependence w.r.t.\ $\Sigma$ in the cost function, and by the linearity of the marginal
constraint. More generally, the same argument shows that, for a generic index set $I$,
\begin{equation}\label{eq:joint_convexity_bis}
W_2^2\left(\int_I \mu_i\d\Theta(i),\int_I\nu_i\d\Theta(i)\right)\leq\int_I W_2^2(\mu_i,\nu_i)\d\Theta(i)
\end{equation}
with $\mu_i$, $\nu_i$ and $\Theta$ probability measures, under appropriate measurability assumptions that are easily
checked in all cases when we are going to apply this formula.

The following result is by now well known, we detail for the reader's convenience some steps of the proof from
\cite{AGS08}.

\begin{proposition}[Existence and stability of optimal maps]\label{optimaT}
Let $D\subset\setR^d$ be a compact set, $\mu,\nu\in\Prob(D)$ with $\mu$ absolutely continuous w.r.t.\ Lebesgue $d$-dimensional measure. Then:
\begin{enumerate}
\item[(a)] there exists a unique optimal transport map $T_\mu^\nu$ from $\mu$ to $\nu$.
\item[(b)] if $\nu_h\to\nu$ weakly in $\Prob(D)$, then $T_\mu^{\nu_h}\to T_\mu^\nu$ in
$L^2(D,\mu;D)$.
\end{enumerate}
\end{proposition}
\begin{proof}
Statement \textit{(a)} is a simple generalization of Brenier's theorem, see for instance \cite[Theorem~6.2.4]{AGS08} for a proof.
The proof of statement \textit{(b)} is typically obtained by combining the stability w.r.t.\ weak convergence of the optimal plans $\nu\mapsto (\Id\times T_\mu^{\nu})$ (see \cite[Proposition~7.1.3]{AGS08}) with a general criterion (see \cite[Lemma~5.4.1]{AGS08}) which allows to deduce convergence in $\mu$-measure of the maps $T_h$ to $T$ from the weak convergence of the plans $(\Id\times T_h)_\#\mu$ to $(\Id\times T)_\#\mu$. 
\end{proof}
%
%
%
\subsection{Transport estimate}

Assume in this section that $D$ is a connected Riemannian manifold, possibly with boundary, whose finite Riemannian volume measure is denoted by $\meas$, with $\dist$ equal to the Riemannian distance.
The estimate from above on $W_2^2$ provided by \autoref{prop:dacorogna-moser} below is closely related to the Benamou-Brenier formula \cite{B00}, \cite[Theorem~8.3.1]{AGS08}, which provides a representation of $W_2^2$ in terms of the minimization of the action $\int_0^1\int_D\abs{{\bf b}_t}^2\d\mu_t\d t$, among all solutions to the continuity equation $\frac{\d}{\d t}\mu_t+{\rm div}({\bf b}_t\mu_t)=0$. It is also related to the Dacorogna-Moser scheme, which 
provides constructively (under suitable smoothness assumptions) a transport map between $\mu_0=u_0\meas$ and 
$\mu_1=u_1\meas$ by solving the PDE
\begin{equation}\label{eq:ambr8}
\left\{\begin{aligned}
& \Delta f = u_1 - u_0 && \text{in $D$}, \\
& \nabla f \cdot n_D = 0 && \text{on $\partial D$}
\end{aligned}\right.
\end{equation}
and then using the flow map of the vector field ${\bf b}_t=u_t^{-1}\nabla f$ at time 1, with $u_t=(1-t)u_1+tu_0$, 
to provide the map. We provide here the estimate without building explictly a coupling, in the spirit of \cite{K10} (see also,
in an abstract setting \cite[Theorem~6.6]{AMS15}), using the duality formula \eqref{eq:ambr5}. This has the advantage to 
avoid smoothness issues and, moreover, uses \eqref{eq:ambr8} only in the weak sense, namely
\[
-\int_D \langle\nabla\phi,\nabla f\rangle\d\meas=\int_D \phi(u_1-u_0)\d\meas\qquad\forall\phi\in\Lip_b(D).
\]

Notice that uniqueness of $f$ in \eqref{eq:ambr8} is obvious, up to additive constants. Existence is guaranteed for $u_i\in L^2(\meas)$ with $\int_D (u_1-u_0)\d\meas=0$  under a spectral gap assumption, thanks to the variational interpretation provided by Lax-Milgram theorem.
Notice also that with the choice ${\bf b}_t=u_t^{-1}\nabla f$ the continuity equation 
$\frac{\d}{\d t} u_t+{\rm div}({\bf b}_t u_t)=0$ holds, in weak form. 

We will also need this definition.

\begin{definition}[Logarithmic mean]
Given $a,b>0$, we define the logarithmic mean
\[
M(a,b) = \frac{a-b}{\log a - \log b} = \left(\int_0^1 \frac1{(1-s)a+sb}\d s\right)^{-1}.
\]
This can be extended to $a,b\geq0$ by continuity, so that $M(a,0)=M(0,b)=M(0,0)=0$.
\end{definition}

\begin{proposition}\label{prop:dacorogna-moser}
Let $u_0,u_1\in L^2(\meas)$ be probability densities with $u_0>0$ $\meas$-a.e.\ in $X$ 
and let $f\in H^{1,2}(D,\meas)$ be any solution to \eqref{eq:ambr8}. Then
\[
W_2^2(u_0\meas,u_1\meas) \leq
\int_0^1 \int_D \frac{\abs{\nabla f}^2}{(1-s)u_0+s u_1} \d\meas \d s =
\int_D \frac{\abs{\nabla f}^2}{M(u_0,u_1)} \d\meas.
\]
\end{proposition} 

\begin{proof} Let $\phi\in\Lip_b(D)$, set $u_s=(1-s)u_0+su_1$ and notice that $u_s>0$ $\meas$-a.e.\ in $X$
for all $s\in [0,1)$. We interpolate, then use Leibniz's rule and \eqref{eq:subsHj} to get
\[
\begin{split}
&\int_D (u_1 Q_1\phi-u_0\phi)\d\meas=
\int_0^1\frac{\d}{\d s}\int_Du_sQ_s\phi\d\meas\d s\\
&=\int_0^1\int_D u_s\frac{\d}{\d s}Q_s\phi+(u_1-u_0)Q_s\phi\d\meas\d s\\
&\le \int_0^1-\frac 12\abs{\nabla Q_s\phi}^2u_s-\langle\nabla f,\nabla Q_s\phi\rangle\d\meas\d s\\
&\leq \frac 12\int_0^1\int_D \frac{\abs{\nabla f}^2}{u_s} \d\meas \d s.
\end{split}
\]
Since $\phi$ is arbitrary, the statement follows from the duality formula \eqref{eq:ambr5}.
\end{proof}

\subsection{Bounds for moments and tails}

In this subsection $(D,\dist)$ is a complete and separable metric space equipped with a Borel probability measure $\meas$. We assume $\diam D<\infty$.

For $n\in\setN_+$, let $X_1,\dotsc,X_n$ be independent and uniformly distributed random variables in $D$, whose common law is $\meas$. 
Let $\mu^n=\frac1n\sum_{i=1}^n \delta_{X_i}$ be the random empirical measure. We define the measures $r^n=\sqrt n(\mu^n-\meas)$, where we use
the natural scaling provided by the central limit theorem. Our goal is to derive upper bounds for the exponential moments $\exp(\lambda\int_D f\d r^n)$ 
and, as a consequence, tail estimates for $\int_D f\d r^n$, related to classical concentration inequalities (Bernstein inequality), see e.g.\ \cite[Lemma~4.3.4]{T14}. For the reader's convenience, we provide a complete proof in the form that we need for our purposes.

\begin{definition}
For $k\in\setN$ and $f\in C_b(D)$, define the $k$-moments $\moment{\plchldr}_k$ by
\[
\moment{f}_k^k = \int_D \left( f(x) - \int_D f\d\meas \right)^k \d\meas(x)
\]
and
\[
\moment{f}_\infty = \norm*{f-\int_Df\d\meas}_{L^\infty(\meas)}.
\]
\end{definition}

Notice that $\moment{f}_0=1$, $\moment{f}_1=0$, $\moment{f}_2\leq\norm{f}_2$ and $\abs*{\moment{f}_{k+2}^{k+2}}\leq \moment{f}_2^2 \moment{f}_\infty^k$.
Moreover, $\moment{\plchldr}_2^2$ is a quadratic form, therefore we introduce also the associated bilinear form
\[
m_2(f,g) = \int_D \left( f(x) - \int_D f\d\meas \right)
	\left( g(x) - \int_D g\d\meas \right) \d\meas(x).
\]
Analogously, we consider also the following quantity
\[
m_4(f,g) = \int_D \left( f(x) - \int_D f\d\meas \right)^2
	\left( g(x) - \int_D g\d\meas \right)^2 \d\meas(x),
\]
so that $m_4(f,f)=\moment{f}_4^4$.

\begin{lemma}[Moment generating function]\label{lem:gen-mom}
Let $f\in C_b(D)$ and $\lambda\in\setR$. Then
\[\begin{split}
\Exp\left[ \exp\left( \lambda \int_D f \d r^n \right) \right] &=
\left\{ \int_D \exp\left[\frac{\lambda}{\sqrt n}\left( f(x) - \int_D f \d\meas \right)\right] \d\meas(x) \right\}^n \\
&= \left( 1 + \sum_{k=2}^\infty \frac{\lambda^k \moment{f}_k^k}{k!n^{k/2}} \right)^n.
\end{split}\]
As a consequence
\begin{equation}\label{eq:gen-mom-upper}
\Exp\left[ \exp\left( \lambda \int_D f \d r^n \right) \right] \leq
\exp\left[ \frac{\lambda^2 \moment{f}_2^2}{2} \exp\left(\frac{\abs\lambda \moment{f}_\infty}{\sqrt n}\right) \right].
\end{equation}
\end{lemma}

\begin{proof}
It is sufficient to show the result for $\lambda=1$. The general statement then follows by taking $\lambda f$ in place of $f$. By the
definition of empirical measure we have
\[\begin{split}
\Exp\left[ \exp\left( \int_D f \d r^n \right) \right] &=
\Exp\left[ \exp\left( \frac1{\sqrt n}\sum_{i=1}^n f(X_i) - \sqrt n \int_D f\d\meas \right) \right]  \\
&= \Exp\left[ \exp\left( \sum_{i=1}^n \frac1{\sqrt n} \left\{ f(X_i) - \int_D f\d\meas \right\} \right) \right] \\
&= \Exp\left[ \exp\left( \frac1{\sqrt n} \left\{ f(X_1) - \int_D f\d\meas \right\} \right) \right]^n \\
&= \left\{ \int_D \sum_{k=0}^\infty \frac1{k!} \left[\frac1{\sqrt n}\left( f(x) - \int_D f \d\meas \right)\right]^k \d\meas(x) \right\}^n. 
\end{split}\]
The equality above gives
\[\begin{split}
\Exp\left[ \exp\left( \int_D f \d r^n \right) \right]
&= \left\{ 1 + \sum_{k=2}^\infty \frac{\moment{f}_k^k}{k!n^{k/2}} \right\}^n
= \left\{ 1 + \sum_{k=0}^\infty \frac{\moment{f}_{k+2}^{k+2}}{(k+2)!n^{k/2+1}} \right\} ^n \\
&\leq \left\{ 1+\frac{\moment{f}_2^2}{2n} \sum_{k=0}^\infty \frac{\moment{f}_\infty^k}{k!n^{k/2}} \right\}^n
= \left\{ 1 + \frac{\moment{f}_2^2}{2n} \exp\left(\frac{\moment{f}_\infty}{\sqrt n}\right) \right\}^n \\
&\leq \exp\left[ \frac{\moment{f}_2^2}{2} \exp\left(\frac{\moment{f}_\infty}{\sqrt n}\right) \right]. \qedhere
\end{split}\]
\end{proof}

\begin{lemma}\label{lem:covariance}
Let $f,g\in C_b(D)$. Then
\begin{equation}\label{eq:am1}
\Exp\left[\left( \int_D f \d r^n \right)^2 \right] = \moment{f}_2^2, \qquad
\Exp\left[\left( \int_D f \d r^n \right)
	\left( \int_D g \d r^n \right) \right] = m_2(f,g),
\end{equation}
and
\[\begin{split}
\Exp\left[ \left( \int_D f \d r^n \right)^2
	\left( \int_D g \d r^n \right)^2 \right] &=
\frac{n-1}n\left[ \moment{f}_2^2\moment{g}_2^2 + 2m_2(f,g)^2 \right] + \frac1n m_4(f,g) \\
&\leq 3\frac{n-1}{n} \moment{f}_2^2\moment{g}_2^2 + \frac1n \moment{f}_4^2\moment{g}_4^2.
\end{split}\]
\end{lemma}

\begin{proof} Since
\[
\Exp\left[ \exp\left( \lambda \int_D f \d r^n \right) \right] =\sum_{k=0}^\infty\frac{\lambda^k}{k!}\Exp\left[\left(\int_D f\d r^n\right)^k\right],
\]
it is sufficient to compute the second and fourth derivatives with respect to $\lambda$ at $\lambda=0$ in the expression for
$\Exp\left[ \exp\left( \lambda \int_D f \d r^n \right) \right]$ provided by \autoref{lem:gen-mom} 
to obtain, respectively, the first identity in \eqref{eq:am1} and
\[
\Exp\left[ \left( \int_D f \d r^n \right)^4 \right] =
	3\frac{n-1}n \moment{f}_2^4 + \frac1n \moment{f}_4^4.
\]
The remaining two identities follow by polarization.
\end{proof}

For $c,\eta>0$, define the function
\begin{equation}\label{eq:defF}
F(c,\eta) = \sup_{\lambda>0} \left\{
\lambda \eta - \frac{\lambda^2}2 \exp(c \lambda)
\right\} > 0.
\end{equation}
Notice that $F(c,\eta)$ is decreasing in $c$, increasing in $\eta$ and that the formula
\[
cF(c,\eta)=\sup_{\lambda>0} \left\{
c\lambda \eta - c\frac{\lambda^2}2 \exp(c \lambda)
\right\}=\sup_{\lambda'>0} \left\{
\lambda' \eta - \frac{(\lambda')^2}{2c} \exp(\lambda')
\right\}
\]
shows that $c F(c,\eta)$ is increasing in $c$.
We will use the function $F$ to estimate the tails of $\int_D f\d r^n$. 

\begin{lemma}[Tail bound]\label{lem:laplace-markov-legendre}
Let $X$ be a real random variable such that, for some $c_1,c_2>0$,
\[
\Exp[\exp(\lambda X)] \leq
\exp\left[ \frac{\lambda^2 c_1}2 \exp(\abs\lambda c_2) \right]
\qquad \forall\lambda\in\setR.
\]
Then for every $\eta\geq 0$ we have
\[
\Pro(\abs X > \eta) \leq
2\exp\left[ - \frac1{c_1} F\left(\frac{c_2}{c_1},\eta\right) \right].
\]
\end{lemma}
\begin{proof}
We have $\Pro(\abs X > \eta) \leq \Pro(X > \eta) + \Pro(X < -\eta)$. For the first term and $\lambda>0$
\[\begin{split}
\Pro(X > \eta) &= \Pro\bigl(\exp(\lambda X) > \exp(\lambda \eta)\bigr) \\
&\leq \Exp[\exp(\lambda X)] \exp(-\lambda \eta) \leq 
\exp\left[ \frac{\lambda^2 c_1}2 \exp(\lambda c_2) - \lambda \eta \right] .
\end{split}\]
Hence
\[
\Pro(X>\eta) \leq
\exp\left[ \inf_{\lambda>0} \left( \frac{\lambda^2 c_1}2 \exp(\lambda c_2) - \lambda \eta \right) \right] 
= \exp\left[ \frac1{c_1} \inf_{\lambda>0} \left\{ \frac{\lambda^2}2 \exp\left(\lambda \frac{c_2}{c_1}\right) - \lambda \eta \right\} \right].
\]
For the other term, we use the fact that $\Pro(X < -\eta) = \Pro(-X > \eta)$ and $-X$ satisfies the same hypothesis.
\end{proof}

\section{Heat semigroup}\label{sec:heat}

In this section we add more structure to $D$, assuming that $(D,\dist)$ is a connected Riemannian manifold (possibly with boundary) endowed with the Riemannian
distance, and that $D$ has finite diameter and volume. Then, we can and will normalize $(D,\dist)$ in such a way that the volume is unitary, and
let $\meas$ be the volume measure of $(D,\dist)$. The typical examples we have in mind are the flat $d$-dimensional torus $\T^d$
and the $d$-dimensional cube $[0,1]^d$, see also \autoref{sec:more_general_spaces} for more general setups.

We denote by $P_t$ the heat semigroup associated to $(D,\dist,\meas)$, with Neumann boundary conditions. In one of the many
equivalent representations, it can be viewed as the $L^2(\meas)$ gradient flow of the Dirichlet energy $\frac 12 \int_D \abs{\nabla f}^2\d\meas$.
Standard results (see for instance \cite{W14}) ensure that $P_t$ is a Markov semigroup, so that it is a contraction semigroup in all
$L^p\cap L^2(\meas)$ spaces, $1\leq p\leq\infty$; thanks to this property it has a unique extension to all $L^p(\meas)$ spaces
even when $p\in [1,2)$. Moreover, the finiteness of volume and boundary conditions ensure that $P_t$ is mass-preserving, i.e.\ $t\mapsto\int_D P_t f\d \meas$ is constant 
in $[0,\infty)$ for all  $f\in L^1(\meas)$ and thus it can be viewed as an operator in the class of probability densities (which
correspond to the measures absolutely
continuous w.r.t.\ $\meas$). More generally, we can use the Feller property (i.e.\ that $P_t$ maps $C_b(D)$ into $C_b(D)$) to define the adjoint 
semigroup $P_t^*$ on the class ${\cal M}$ of Borel measures in $D$ with finite total variation by
\[
\int_D f\d P_t^*\mu=\int_D P_tf\d\mu
\]
and to regularize with the aid of $P_t^*$ singular measures to absolutely continuous measures, under appropriate additional
assumptions on $P_t$. Since $P_t$ is selfadjoint, the operator $P^*_t$ can also 
be viewed as the extension of $P_t$ from $L^1(\meas)$ to ${\cal M}$.

We denote by $p_t(x,y)$ the transition probabilities of the semigroup, characterized by the formula
\[
P_t f(x) = \int_D p_t(x,y) f(y) \d\meas(y),
\]
so that 
\[
P_t^*\delta_{x_0}=p_t(x_0,\plchldr)\meas\qquad\text{for all $x_0\in D$, $t>0$}.
\]

We denote by $\Delta$ the infinitesimal generator of $P_t$, namely the extension of the Laplace-Beltrami operator on $D$. Besides
the ``qualitative'' properties of $P_t$ mentioned above, our proof depends on several quantitative estimates related to $P_t$.

\paragraph{\bf Quantitative estimates on $P_t$.}
We assume throughout the validity of the following properties:
there are positive constants $d$, $\Csg$, $\Cuc$, $\Cg$, $\Cr$, $\Cdr$  and $K$ such that
\begin{enumerate}[leftmargin=*,widest=(UC)]
\item[(SG)] spectral gap: $\norm{P_t f}_2 \leq e^{-\Csg t} \norm{f}_2$ for any $f$ with $\int_D f\d\meas=0$,
\item[(UC)] ultracontractivity: $\abs{p_t(x,y)-1} \leq \Cuc t^{-d/2}$ for all $x,y\in D$,
\item[(GE)] gradient estimate: $\Lip\bigl(p_t(x,\plchldr)\bigr) \leq \Cg t^{-(d+1)/2}$,
\item[(RT)] Riesz transform bound:
\[
\int_D \abs{\nabla f}^4 \d\meas \leq \Cr \int_D \abs*{(-\Delta)^{1/2}f}^4 \d\meas,
\]
\item[(DR)] dispersion rate: $\int_D \dist^2(x,y)p_t(x,y)\d\meas(y)\leq \Cdr t$,
\item[(GC)] gradient contractivity: $\abs{\nabla P_t f}\leq e^{Kt}P_t\abs{\nabla f}$.
\end{enumerate}

In the sequel, since many parameters and constants will be involved, in some statements we call a constant \textit{geometric}
if it depends only on $D$ through $\Csg$, $\Cuc$, $\Cg$, $\Cr$, $\Cdr$ and $K$. Notice that (GC) encodes a lower bound
on Ricci curvature, see for instance \cite{W11}.
Let us draw now some easy consequences of these assumptions.

Spectral gap implies that for $f\in L^2(D,\meas)$ with $\Delta f\in L^2(D,\meas)$ we have the representation
\begin{equation}\label{eq:representation-formula}
f(x) = \int_D f\d\meas + \int_0^\infty (P_t \Delta f)(x) \d t.
\end{equation}
Ultracontractivity entails that $P_t:L^1\to L^\infty$ continuously for $t>0$, because
\[
\abs{P_tf(x)} \leq \abs*{\int_D p_t(x,y)f(y)\d\meas(y) - \int_Df(y)\d\meas(y)} + \norm{f}_1 \leq
(\Cuc t^{-d/2} + 1) \norm{f}_1.
\]
Hence, by interpolation $P_t : L^p\to L^q$ for any $1\leq p\leq q\leq\infty$ with norms bounded from above by geometric constants. If $p=1$, by approximation
we also get $P^*_t:{\cal M}\to L^q$ continuously for $1\leq q\leq\infty$. Notice also that
\begin{equation}\label{eq:m2-minfty-ineq}
\moment{p_t(\plchldr,y)}_2^2 \leq \moment{p_t(\plchldr,y)}_\infty \leq \Cuc t^{-d/2}
\end{equation}
because
\[\begin{split}
\moment{p_t(\plchldr,y)}_2^2 &=
\int_D \bigl(p_t(x,y)-1\bigr)^2 \d\meas(x) =
\int_D p_t(x,y)\bigl(p_t(x,y)-1\bigr) \d\meas(x) \\
&\leq \moment{p_t(\plchldr,y)}_\infty \int_D p_t(x,y) \d\meas(x) =
\moment{p_t(\plchldr,y)}_\infty \leq \Cuc t^{-d/2}.
\end{split}\]
Writing $\mu=\int_D\delta_x\d \mu(x)$ and (DR) in the form $W_2^2(P_t^*\delta_x,\delta_x)\leq\Cdr t$,
from the joint convexity of $W_2^2$ \eqref{eq:joint_convexity_bis} we obtain
\[
W_2^2(P^*_t\mu,\mu) \leq \Cdr t.
\]
By duality, see \cite{K10}, the gradient contractivity property leads to contractivity w.r.t.\ $W_2$ distance
\[
W_2^2(P_t^*\mu,P_t^*\nu) \leq e^{2Kt} W_2^2(\mu,\nu).
\]
Moreover, it implies that for some geometric constant $C$ we have
\begin{equation}\label{eq:elliptic-regularity}
\norm{\nabla f}_\infty \leq C \norm{\Delta f}_\infty
\end{equation}
for every $f\in L^2(D,\meas)$ with $\Delta f\in L^\infty(D,\meas)$.
In fact (GC) gives the reverse Poincaré inequality \cite{W11}
\[
\abs{\nabla P_t g}^2 \leq
\frac{K}{e^{2Kt}-1} \left[ P_t(g^2) - (P_tg)^2 \right]
\leq \frac{K}{e^{2Kt}-1} \norm{g}_\infty^2,
\]
hence the bound $\norm{\nabla P_t g}_\infty \leq c t^{-1/2} \norm{g}_\infty$ for $t\in(0,1]$ and some geometric constant $c$.
Using the representation formula \eqref{eq:representation-formula} and the previous estimate with $g=\Delta f$ and $g=P_{t-1}\Delta f$ we obtain \eqref{eq:elliptic-regularity} as
\[
\begin{split}
\norm{\nabla f}_\infty &\leq
	\int_0^\infty \norm{\nabla P_t\Delta f}_\infty \d t
= \int_0^2 \norm{\nabla P_t\Delta f}_\infty \d t
	+ \int_2^\infty \norm{\nabla P_1(P_{t-1}\Delta f)}_\infty \d t \\
&\leq c \left(\norm{\Delta f}_\infty \int_0^2 t^{-1/2} \d t +
	\norm{P_1}_{L^2\to L^\infty} \int_2^\infty \norm{P_{t-2}\Delta f}_2 \d t
	\right) \\
&\leq c \left( 2\sqrt2 + \norm{P_1}_{L^2\to L^\infty}
	\int_2^\infty e^{-\Csg(t-2)} \d t\right) \norm{\Delta f}_\infty.
\end{split}
\]
 
In the following lemma we collect some more consequences of the gradient contractivity.
\begin{lemma}\label{lem:abstract-semigroup-hopf-lax}
For every $s \ge  0$ and $g \in C_b(D)$ one has
\begin{equation}\label{eq:abstract-semigroup-hopf-lax}
  \min g \le -\log \left(  P_s e^{-g } \right) \le \max g, \quad \text{hence}\quad  \norm{ -\log \left(  P_s e^{-g } \right)}_\infty \le \norm{g}_\infty,
\end{equation}
\begin{equation}\label{eq:abstract-semigroup-hopf-lax-gradient} 
\norm{ \nabla  \log \left(  P_s e^{-g } \right) }_\infty \le e^{K^-s}\norm{ \nabla g}_\infty.
\end{equation}
\end{lemma}

\begin{proof}
Write $G = e^{-g}$. Inequality \eqref{eq:abstract-semigroup-hopf-lax} follows from the fact that $P_s$ is Markov and the inequalities
$e^{-\max g} \le  G \le e^{-\min g}$. 
In order to prove \eqref{eq:abstract-semigroup-hopf-lax-gradient} we use (GC) to get
\[ \abs{ \nabla  \log \left(  P_s e^{-g } \right)} = 
\frac{ \abs{ \nabla  P_s e^{-g }}} { P_s e^{-g }} \le e^{K^-s}\frac{ P_s\left( \abs{\nabla g} e^{-g } \right) }{ P_s e^{-g }} \le e^{K^-s}\norm{ \nabla g}_\infty.\qedhere\]

\end{proof}

\begin{lemma}[Viscous Hamilton-Jacobi]
Assume that $D$ is a compact Riemannian manifold without boundary. Let $\sigma >0$, $f\in C(D)$, and define, for $t \geq 0$,
\[
\phi_t^\sigma = -\sigma \log\left( P_{(\sigma t)/2} e^{-f/\sigma} \right).
\]
Then $\phi_t^\sigma \in C\bigl([0,+\infty)\times D\bigr)\cap C^\infty\bigl((0,+\infty)\times D)$ solves
\begin{equation}\label{eq:hj} \left\{\begin{aligned}
\partial_t \phi_t^\sigma &= -\frac{\abs{\nabla\phi_t^\sigma}^2}2 + \frac{\sigma}{2} \Delta\phi_t^\sigma
	&& \text{in $(0,+\infty)\times D$}, \\
\phi_0^\sigma &= f && \text{in $D$}.
\end{aligned}\right.\end{equation}
Moreover
\begin{gather}
 \label{eq:hj-uniform} \min f \le \phi^\sigma_t \le \max f, \qquad \norm{\phi_t^\sigma}_\infty \leq \norm{f}_\infty,\\
 \label{eq:hj-gradient} \qquad \norm{\nabla \phi_t^\sigma}_\infty \leq e^{K^-t}\norm{\nabla f}_\infty, \\ 
\label{eq:hj-Delta} \norm{(\Delta\phi_t^\sigma)^+}_\infty \leq \norm{(\Delta f)^+}_\infty+\frac{e^{2K^-t}-1}{2}\norm{\nabla f}_\infty^2,\\
 \label{eq:hj-distance}\phi_1^\sigma(y) - \phi_0^\sigma(x) \leq \frac{\dist(x,y)^2}2
	+ \frac{ \sigma}{2} \norm{(\Delta f)^+}_\infty+\frac{\sigma}{4}(e^{2K^-}-1)\norm{\nabla f}_\infty^2\\
\label{eq:hj-energy} \int_D (\phi_0^\sigma - \phi_1^\sigma) \d\meas \leq
	e^{\norm{(\Delta f)^+}_\infty+\frac{1}{2}(e^{2K^-}-1)\norm{\nabla f}_\infty^2} \int_D \frac{\abs{\nabla f}^2}2 \d\meas.
\end{gather}
\end{lemma}

\begin{proof} The smoothness of $\phi^\sigma_t$ for positive times follows by the chain rule and standard (linear) parabolic theory.
To check that $\phi^\sigma$ solves \eqref{eq:hj}, it is sufficient to compare
\[
\partial_t \phi^\sigma_t = -\frac{ \sigma^2}{2} \frac{ \Delta P_{(\sigma t)/2} e^{-f/\sigma}}{P_{(\sigma t)/2} e^{-f/\sigma}}
\]
with the terms arising from the application of the diffusion chain rule
\begin{equation}\begin{split} \frac{\sigma}{2} \Delta \phi^\sigma_t & =- \frac{\sigma^2}{2} \Delta\log\left( P_{(\sigma t)/2} e^{-f/\sigma} \right) \\
& = -\frac{ \sigma^2}{2} \frac{ \Delta P_{(\sigma t)/2} e^{-f/\sigma}}{P_{(\sigma t)/2} e^{-f/\sigma}} +
	\frac{\sigma^2}{2} \abs*{ \frac{ \nabla P_{(\sigma t)/2} e^{-f/\sigma}}
	{ P_{(\sigma t)/2} e^{-f/\sigma} } }^2 \\
&=  \partial_t \phi^\sigma_t + \frac{1}{2} \abs{ \nabla \phi^\sigma_t}^2.
\end{split}\end{equation}

Inequalities \eqref{eq:hj-uniform} and \eqref{eq:hj-gradient} 
follow in a straightforward way, respectively from \eqref{eq:abstract-semigroup-hopf-lax} and 
\eqref{eq:abstract-semigroup-hopf-lax-gradient} 
of \autoref{lem:abstract-semigroup-hopf-lax}, with $s = (\sigma t)/2$ and $g = f/\sigma$.

To prove \eqref{eq:hj-Delta} we use Bochner's inequality 
\[
\Delta \frac{\abs{\nabla\phi_t^\sigma}^2}{2}\geq K\abs{\nabla\phi_t^\sigma}^2+\langle\nabla\phi_t^\sigma,\nabla\Delta\phi_t^\sigma\rangle,
\]
which encodes the bound from below on Ricci curvature, and, setting $\xi_t=\Delta \phi^\sigma_t$, we get
\[ 
\partial_t \xi_t\leq -K\abs{\nabla\phi_t^\sigma}^2 -\langle\nabla\phi_t^\sigma,\nabla\xi_t\rangle+\frac\sigma 2\Delta\xi_t
\leq
K^-e^{2K^- t}\norm{\nabla f}_\infty^2 -\langle\nabla\phi_t^\sigma,\nabla\xi_t\rangle+\frac\sigma 2\Delta\xi_t
\]
which, by the maximum principle, leads to \eqref{eq:hj-Delta}.

To prove \eqref{eq:hj-distance}, let $\gamma \in C^1( [0,1], D)$, with $\gamma(0)=x$, $\gamma(1) = y$, and compute
\begin{equation}\begin{split}
 \frac{\d}{\d t} \phi^\sigma_t (\gamma(t))  & = \left(\partial_t\phi_t^\sigma\right) \left(\gamma(t)\right) + \left\langle \left(\nabla \phi^\sigma_t\right) \left(\gamma(t)\right), \dot\gamma(t) \right\rangle \\
 & = - \frac 1 2 \abs{ \nabla \phi^\sigma_t \left(\gamma(t)\right) }^2  + \frac{\sigma}{2} \Delta \phi^\sigma_t \left(\gamma(t)\right) +  \left\langle \left(\nabla \phi^\sigma_t\right)\left(\gamma(t)\right), \dot\gamma(t)\right\rangle \\
 & \le \frac 1 2 \abs{ \dot \gamma(t)} ^2 + \frac{\sigma}{2} \norm{ \left( \Delta \phi^\sigma_t \right)^+ }_\infty.
\end{split} \end{equation}
Integrating over $t \in (0,1)$ and using \eqref{eq:hj-Delta}, we obtain 
\[\phi^\sigma_1( y)  - \phi^\sigma_0(x)   \le \frac 1 2\int_0^1\abs{ \dot \gamma(t)} ^2 \d t + \frac{\sigma}{2} \norm{(\Delta f)^+}_\infty+
\frac{\sigma}{4}(e^{2K^-}-1)\norm{\nabla f}_\infty^2,\]
which yields \eqref{eq:hj-distance} after we take the infimum with respect to $\gamma$.

To show \eqref{eq:hj-energy}, we notice first that 
\begin{equation}\label{eq:ambr12}
\begin{split}
\int_D (\phi_0^\sigma - \phi_1^\sigma) \d\meas & = - \int_D \int_0^1 \partial_t  \phi_t^\sigma \d t \d\meas \\
& = \frac{1}{2} \int_0^1 \int_D \abs{ \nabla \phi_t^\sigma }^2  \d\meas - \frac{\sigma}{2} \int_0^1 \int_D  \Delta \phi_t^\sigma \d \meas \\
& = \frac{1}{2} \int_0^1 \int_D \abs{ \nabla \phi_t^\sigma }^2  \d\meas,
\end{split}\end{equation}
where the second term vanishes because $D$ is without boundary.
For $t \in (0,1)$, one has
\begin{equation}\label{eq:ambr13}
\begin{split}
 \frac{\d}{\d t} \int_D \abs{ \nabla \phi_t^\sigma }^2 \d \meas & = - \frac{\d}{\d t}
 \int_D \left( \Delta \phi_t^\sigma\right) \phi_t^\sigma  \d \meas  = - 2 \int_D \left( \Delta  \phi_t^\sigma\right) \partial_t \phi_t^\sigma  \d \meas \\
 & =  \int_D \left( \Delta  \phi_t^\sigma\right) \abs{ \nabla \phi_t^\sigma }^2 - \sigma \int_D  \left( \Delta  \phi_t^\sigma\right)^2 \d \meas \\
 & \le \norm{(\Delta \phi_t^\sigma)^+}_\infty \int_D \abs{ \nabla \phi_t^\sigma }^2 \d \meas.
\end{split}
\end{equation}
Combining \eqref{eq:ambr12} and \eqref{eq:ambr13} and taking into account the estimate \eqref{eq:hj-Delta} on $\Delta\phi_t^\sigma$,
inequality \eqref{eq:hj-energy} follows by Gronwall's inequality.\qedhere
\end{proof}

\begin{corollary}[Dual potential]\label{cor:good-enough-potential}
Assume that $D$ is a compact Riemannian manifold without boundary. 
For every Lipschitz function $f$ with $\norm{(\Delta f)^-}_\infty< \infty$, there exists $g\in C_b(D)$ such that
\[
f(x)+g(y) \leq \frac{\dist(x,y)^2}2, \quad
\int_D (f+g)\d\meas \geq
	- e^{\norm{(\Delta f)^-}_\infty+\frac{1}{2}(e^{2K^-}-1)\norm{\nabla f}_\infty^2} \int_D \frac{\abs{\nabla f}^2}2 \d\meas.
\]
\end{corollary}

\begin{proof}
For $\sigma>0$, consider the functions $g^\sigma = \phi_1^\sigma$ solving the initial value problem \eqref{eq:hj} with $f$ replaced by $-f$. Inequalities \eqref{eq:abstract-semigroup-hopf-lax} and \eqref{eq:abstract-semigroup-hopf-lax-gradient} entail that $g^\sigma$ are uniformly bounded in the space of Lipschitz functions: as $\sigma \to 0$, we can extract a subsequence $(g^{\sigma_h})$ pointwise converging to some bounded Lipschitz function $g$. Inequality \eqref{eq:hj-distance} gives in the limit the first inequality of the thesis, while \eqref{eq:hj-energy} yields the second one, by dominated convergence.
\end{proof}

\begin{remark} [On the equality $g=Q_1(-f)$] Recall that the theory of viscosity solutions \cite{CL83}, \cite{BC97} is specifically designed to deal with equations, as the Hamilton-Jacobi equations, for which the distributional point of view fails. This theory can be carried out also on manifolds, see \cite{F} for a nice presentation of this subject. Since one can prove (using also apriori estimates on the time derivatives, arguing as in \autoref{cor:good-enough-potential}) the existence of a function $\phi_t$, uniform limit of a subsequence of $\phi_t^\sigma$, since classical solutions are viscosity solutions and since locally uniform limits of viscosity solutions are viscosity solutions, the function $\phi_t$ is a viscosity solution to the HJ equation $\partial_t u+\frac 12|\nabla u|^2=0$. Then, if the initial condition is $-f$, the uniqueness theory of first order viscosity solutions applies, and gives that $\phi_t$ is precisely given by
\[
\inf_{y\in D}\frac{\dist^2(x,y)}{2t}-f(y).
\]
Setting $t=1$, this argument proves that actually the function $g$ of  \autoref{cor:good-enough-potential} coincides with $Q_1(-f)$, and that there is full convergence as $\sigma\to 0$ (see also \cite{C03} for a proof of the convergence, in Euclidean spaces, based on the theory of large deviations). We preferred a more elementary and self-contained presentation, because the weaker statement $g\leq Q_1(-f)$ provided by the Corollary is sufficient for our purposes, and because our argument works also in the more abstract setting described in \autoref{sec:more_general_spaces} (in which neither large deviations nor theory of viscosity solutions are yet available), emphasizing the role played by the lower Ricci curvature bounds.
\end{remark}

\subsection{Density fluctuation bounds}

Recalling the notation $\mu^n=\frac1n\sum_{i=1}^n \delta_{X_i}$, $r^n=\sqrt{n}(\mu^n-\meas)$, we now define
our regularized empirical measures.

\begin{definition}[Regularized empirical measures]\label{def:reg-emp-meas}
For $t\geq 0$ define
\[
\mu^{n,t} = P_t^* \mu^n, \qquad r^{n,t}\meas = P^*_t r^n = \sqrt{n}(\mu^{n,t} - \meas),
\]
so that for $t>0$ one has
\[
r^{n,t}(y)=\int_D (p_t(\cdot,y)-1) \d r^n.
\]
\end{definition}

The goal of this subsection is to collect apriori estimates on the deviation of $r^{n,t}$ from $0$.

\begin{lemma}[Pointwise bound]\label{lem:rt-tail-bound}
For $y\in D$ and $\eta>0$ one has
\begin{equation}\label{eq:rt-tail-bound}
\Pro\left( \frac{\abs{r^{n,t}(y)}}{\sqrt n} > \eta \right) \leq
2 \exp\left( - \frac{nt^{d/2}}{2\Cuc} F(1,\eta) \right),
\end{equation}
where $F$ is defined in \eqref{eq:defF}.
\end{lemma}

\begin{proof}
Consider the random variable $X = r^{n,t}(y)/\sqrt n=\int_D p_t(\cdot,y)/\sqrt{n}\d r^n$.
By \eqref{eq:gen-mom-upper} with $f=p_t(\cdot,y)/\sqrt{n}$ we have
\[\begin{split}
\Exp[\exp(\lambda X)] &\leq
\exp\left[ \frac{\lambda^2 \moment{p_t(\plchldr,y)}_2^2}{2n}
	\exp\left( \abs\lambda \frac{\moment{p_t(\plchldr,y)}_\infty}{n} \right)\right] \\
&\leq \exp\left[ \frac{\lambda^2}2\cdot \frac{\Cuc}{nt^{d/2}}
	\exp\left( \abs\lambda \frac{\Cuc}{nt^{d/2}} \right)\right],
\end{split}\]
where in the second inequality we used \eqref{eq:m2-minfty-ineq}.
Then \autoref{lem:laplace-markov-legendre} with $c_1=c_2=\Cuc/(nt^{d/2})$ implies \eqref{eq:rt-tail-bound}.
\end{proof}

\begin{lemma}[Deterministic bound]\label{lem:deterministic-oscillation}
With probability $1$ one has
\[
\frac{\abs{r^{n,t}(y) - r^{n,t}(z)}}{\sqrt n} \leq
\frac{2\Cg}{t^{(d+1)/2}} \dist(y,z).
\]
\end{lemma}

\begin{proof} Using (GE) and the fact that the total variation of the measures $r^n$ is $2\sqrt{n}$, we get
\[\begin{split}
\frac{\abs{r^{n,t}(y) - r^{n,t}(z)}}{\sqrt n} &=
\abs*{ \int_D \bigl(p_t(x,y) - p_t(x,z)\bigr) \frac{\d r^n(x)}{\sqrt n}}  \\
&\leq \dist(y,z)\int_D \Lip\bigl(p_t(x,\plchldr)\bigr)\frac{\d\abs{r^n}(x)}{\sqrt n}
\leq \frac{2\Cg}{t^{(d+1)/2}} \dist(y,z). \qedhere
\end{split}\]
\end{proof}

We shall need another geometric function related to $D$.

\begin{definition}[Minimal $\delta$-cover]
In the sequel, for $\delta>0$ we denote by $N_D(\delta)$ be smallest cardinality of a $\delta$-net of $D$, namely
a set whose closed $\delta$-neighbourhood contains $D$.
\end{definition}

\begin{proposition}[Uniform bound, $d=1$]\label{prop:d1-scaling}
Assume that ultracontractivity holds with $d=1$ and that $N_D(\delta)\leq \max\{1,C_D\delta^{-1}\}$ for all
$\delta>0$. Then there exists a constant $C=C(\Cg,C_D)$ with the following property: for all
$\eta\in (0,1)$, $q\in(0,1)$ and $\eta^{-1} n^{-2q}\leq t\leq 4\Cg C_D$, we have
\[
\Pro\left( \sup_{y\in D} \frac{\abs{r^{n,t}(y)}}{\sqrt n} > \eta \right) \leq
C \exp\left(-\gamma n^{1-q}\right)
\]
with $\gamma=\gamma(\eta,\Cuc)$ and $n\geq n(\eta,q,\Cuc)$.
\end{proposition}

\begin{proof}
We pick $\delta = \frac{\eta}{4\Cg}t$, so that, by \autoref{lem:deterministic-oscillation}, with probability $1$ we have
\begin{equation}\label{eq:amb3}
\frac{\abs{r^{n,t}(y) - r^{n,t}(z)}}{\sqrt n} \leq \frac\eta2\qquad\text{for any $y,z\in D$ with $\dist(y,z)\leq\delta$.}
\end{equation}
Let $T$ be a minimal $\delta$-net. Then the condition $t\leq 4\Cg C_D$ implies $C_D\delta^{-1}\geq 1$, hence 
\[
\abs{T} \leq C_D \delta^{-1} = \frac{4\Cg C_D}{\eta} t^{-1} \leq 4\Cg C_D n^{2q}.
\]
From an application of \autoref{lem:rt-tail-bound} with $\eta/2$ instead of $\eta$ we get
\[\begin{split}
\Pro \left( \sup_{y\in T} \frac{\abs{r^{n,t}(y)}}{\sqrt n} > \frac\eta2 \right) &\leq
2\abs{T} \exp\left( - \frac{nt^{1/2}}{2\Cuc} F(1,\eta/2) \right)  \\
&\leq 8\Cg C_D \exp\left( 2q\log n - \frac{n^{1-q}}{2\eta^{1/2}\Cuc} F(1,\eta/2) \right) \\
&\leq 8\Cg C_D \exp\left(-\gamma n^{1-q}\right),
\end{split}\]
where the last inequality holds with $\gamma=F(1,\eta/2)/(4\eta^{1/2}\Cuc)$ and $n\geq n(\eta,q,\Cuc)$, absorbing
the logarithm $\log n$ into the power $n^{1-q}$.
We conclude since
\[
\Pro\left( \sup_{y\in D} \frac{\abs{r^{n,t}(y)}}{\sqrt n} > \eta \right) \leq
\Pro\left( \sup_{y\in T} \frac{\abs{r^{n,t}(y)}}{\sqrt n} > \frac \eta2 \right). \qedhere
\]
\end{proof}

\begin{proposition}[Uniform bound, $d=2$]\label{prop:d2-scaling}
Assume that ultracontractivity holds with $d=2$ and that $N_D(\delta)\leq\max\{1,C_D\delta^{-2}\}$ 
for every $\delta>0$. Then there exists a constant $C=C(\Cg,C_D)$ with the following property: for all
$\eta>0$ there exists $\gamma=\gamma(\eta,\Cuc)$ such that
\[
\Pro\left( \sup_{y\in D} \frac{\abs{r^{n,t}(y)}}{\sqrt n} > \eta \right) \leq \frac Cn
\]
holds for $(16C_D\Cg^2)^{1/3}\geq t\geq \gamma n^{-1}\log n$ and $n\geq n(\eta,\Cuc)$.
\end{proposition}

\begin{proof} Given $\eta>0$, we choose $\gamma$ in such a way that $\gamma F(1,\eta/2)/(2\Cuc)=4$. 
Then, we define $n(\eta,\Cuc)$ in such a way that $\gamma\log n\geq \eta^{-2/3}$ for $n\geq n(\eta,\Cuc)$.

We pick $\delta = \frac{\eta}{4\Cg}t^{3/2}$, so that, by \autoref{lem:deterministic-oscillation}, with probability $1$ we have
\eqref{eq:amb3}. Let $T$ be a minimal $\delta$-net. Then the condition $t^3\leq 16C_D\Cg^2$ implies $C_D\delta^{-2}\geq 1$, 
\[
\abs{T} \leq C_D \delta^{-2} = \frac{16C_D \Cg^2}{\eta^2} t^{-3} \leq 16C_D\Cg^2 n^3,
\]
where we used also the inequality $t\geq\gamma n^{-1}\log n\geq \eta^{-2/3}/n$.
From an application of \autoref{lem:rt-tail-bound} with $\eta/2$ instead of $\eta$ we get
\[\begin{split}
\Pro \left( \sup_{y\in T} \frac{\abs{r^{n,t}(y)}}{\sqrt n} > \frac \eta2 \right) &\leq
2\abs{T} \exp\left( - \frac{nt}{2\Cuc} F(1,\eta/2) \right) \\
&\leq 32 C_D\Cg^2 \exp\left(3\log n - \gamma \frac{F(1,\eta/2)}{2\Cuc}\log n \right) .
\end{split}
\]
Our choice of $\gamma$ then gives
\[
\Pro \left( \sup_{y\in T} \frac{\abs{r^{n,t}(y)}}{\sqrt n} > \frac \eta2 \right) 
\leq \frac{32 C_D\Cg^2}{n}.
\qedhere
\]
\end{proof}

We now report some estimates on the logarithmic mean.

\begin{lemma}\label{lem:logmean-inequalities}
For $a,b\geq0$ and $q>0$ we have
\[
q(ab)^{q/2} \frac{a-b}{a^q-b^q} \leq M(a,b) \leq
q\frac{a^q+b^q}{2}\cdot\frac{a-b}{a^q-b^q}.
\]
\end{lemma}

\begin{proof}
It is known that
\begin{equation}\label{eq:ambr2}
\sqrt{ab} \leq M(a,b) \leq \frac{a+b}{2}.
\end{equation}
The thesis follows by applying these inequalities to $a^q$ and $b^q$.
\end{proof}

In the following lemma we estimate the logarithmic mean of the densities of $\mu^{n,t,c}$ obtained by
a further regularization, i.e.\ by adding to $\mu^{n,t}$ a small multiple of $\meas$.

\begin{lemma}[Integral bound]\label{lem:log-mean-convergence}
Define $\mu^{n,t,c}=(1-c)\mu^{n,t}+c\meas$ and let $u^{n,t,c}=(1-c)u^{n,t}+c$ be its probability density, with $c=c(n)\in (0,1]$.
If $t^{d/2}=t^{d/2}(n)\geq\gamma n^{-1}\log n$ and $nc(n)\to\infty$ as $n\to\infty$, then
\[
\lim_{n\to\infty}
\Exp\left[
	\int_D \left( M(u^{n,t,c},1)^{-1}-1 \right)^2 \d\meas
	\right] = 0.
\]
\end{lemma}

\begin{proof}
Fix $x\in D$ and $\eta\in(0,1)$. By \autoref{lem:rt-tail-bound} we have
\[
\Pro\left( \abs{u^{n,t}(x)-1} > \eta \right) \leq 2n^{-q},
\]
where $q\in (0,1)$ depends only on $d$, $\Cuc$, $\gamma$ and $\eta$.
In the event $\{\abs{u^{n,\eps}(x)-1} > \eta\}$, using the first inequality in \autoref{lem:logmean-inequalities} we can estimate
the squared difference with the sum of squares to get
\[
\left( M(u^{n,t,c}(x),1)^{-1}-1 \right)^2 \leq
\frac 2{q^2} \cdot \frac1{u^{n,t,c}(x)^q}
	\left( \frac{u^{n,t,c}(x)^q-1}{u^{n,t,c}(x)-1} \right)^2 + 2
\leq \frac 2{q^2c^q} + 2.
\]
In the complementary event $\{\abs{u^{n,t}(x)-1} \leq \eta\}$,
we have $\abs{u^{n,t,c}(x)-1}\leq(1-c)\eta\leq\eta$ and, expanding the squares and 
using both inequalities in \eqref{eq:ambr2}, we get
\[
\left( M(u^{n,t,c}(x),1)^{-1}-1 \right)^2 \leq
\frac{1}{u^{n,t,c}(x)} - \frac{4}{u^{n,t,c}(x)+1} + 1 
\leq \frac{1}{1-\eta} - \frac{4}{2+\eta} + 1.
\]
Therefore
\[
\Exp\left[
	\int_D \left( M(\mu^{n,t,c},1)^{-1}-1 \right)^2 \d\meas
	\right] \leq
2n^{-q}\left(\frac 2{q^2c(n)^q}+2\right)
	+ \frac{1}{1-\eta} - \frac{4}{2+\eta} + 1,
\]
hence the growth condition on $c$ gives
\[
\limsup_{n\to\infty} \Exp\left[
	\int_D \left( M(u^{n,t,c},1)^{-1}-1 \right)^2 \d\meas
	\right] \leq
\frac{1}{1-\eta} - \frac{4}{2+\eta} + 1.
\]
Letting $\eta\to0$ we obtain the result.
\end{proof}

\subsection{Energy estimates}

Retaining \autoref{def:reg-emp-meas} of $r^{n,t}$ from the previous subsection,
here we derive energy bounds for the solutions to the following random PDE:
\begin{equation}\label{eq:energy-estimate}
\left\{\begin{aligned}
&\Delta f^{n,t} = r^{n,t} && \text{in $D$}, \\
& \nabla f^{n,t} \cdot n_D = 0 && \text{on $\partial D$}
\end{aligned}\right.
\end{equation}
which are uniquely determined up to a (random) additive constant.
As we will see (particularly in \autoref{sec:more_general_spaces}), these estimates involve either the trace 
of $\Delta$ or sums indexed by the spectrum $\sigma(\Delta)$ (which contains $\{0\}$ and, by the spectral gap
assumption, satisfies $\sigma(\Delta)\subset (-\infty,-\Csg^2]\cup\{0\}$); 
it is understood that the eigenvalues in these sums are counted with multiplicity.

We recall the so-called trace formula
\begin{equation}\label{eq:trace-formula}
\int_D p_s(x,x) \d\meas(x) = \sum_{\lambda\in\sigma(\Delta)} e^{s\lambda}
\end{equation}
which follows easily by integration of the representation formula 
\[
p_s(x,y)=\sum_{\lambda\in\sigma(\Delta)}e^{s\lambda}u_\lambda(x)u_\lambda(y),
\]
where $\{u_\lambda\}_{\lambda\in \sigma(\Delta)}$ is an $L^2(\meas)$ orthonormal basis of eigenvalues of $\Delta$.

The following expansion \eqref{eq:expansion} of the trace formula as $s\to 0$ will be useful. In this paper we will only
use the leading term in \eqref{eq:expansion}.

\begin{proposition}[Expansion of the trace formula] \label{prop:expansion}
Let $D$ be a bounded Lipschitz domain in $\setR^n$ with unit volume. Then
\begin{equation}\label{eq:expansion}
\int_D p_s(x,x) \d\meas(x) = (4\pi s)^{-d/2}
\left(1+\frac{\sqrt{\pi s}}2\haus^{n-1}(\partial D)+o(\sqrt{s})\right)
\quad\text{as $s\to 0$.}
\end{equation}
The same holds if $D$ is a smooth, compact $d$-dimensional Riemannian manifold with a smooth boundary (possibly
empty).
\end{proposition}
\begin{proof} The first statement is proved in \cite{B93}. The second one, also with additional terms in the expansion, 
in \cite{MKS67}.
\end{proof}

\begin{lemma}[Representation formula] \label{lem:rep}
Let $f^{n,t}$ be the solution to \eqref{eq:energy-estimate}. For all $t>0$ one has
\begin{equation}\label{eq:energy2-n-eps}
\Exp\left[\int_D \abs{\nabla f^{n,t}}^2 \d\meas \right] =
2\int_t^\infty\left(\int_D p_{2s}(x,x) \d\meas(x) - 1\right)\d s
= - \sum_{\lambda\in\sigma(\Delta)\setminus\{0\}} \frac{e^{2t\lambda}}{\lambda}. 
\end{equation}
\end{lemma}
\begin{proof} Using the representation formula $g=-\int_0^\infty P_s\Delta g\d s$ with $g=f^{n,t}$ 
we get $f^{n,t}=-\int_0^\infty P_sr^{n,t} \d s$, so that
\begin{equation}\label{eq:energy-1d-representation}
\begin{split}
\int_D \abs{\nabla f^{n,t}}^2 \d\meas &=
- \int_D f^{n,\eps}\Delta f^{n,t} \d\meas
= \int_D \left(\int_0^\infty P_s r^{n,t} \d s\right) r^{n,t} \d\meas \\
&= \int_0^\infty \int_D P_sr^{n,t} r^{n,t} \d\meas \d s
= \int_0^\infty \int_D P_{s/2}r^{n,t} P_{s/2}r^{n,t} \d\meas \d s \\
&= 2 \int_t^\infty \int_D (P^*_sr^n)^2 \d\meas \d s.
\end{split}
\end{equation}
Now, notice that the symmetry and semigroup properties of the transition probabilities give
\[\begin{split}
\int_D \int_D p_s(x,y)^2 \d\meas(x) \d\meas(y)
&= \int_D\int_D p_s(x,y)p_s(y,x)\d\meas(y) \d\meas(x) \\
&= \int_D p_{2s}(x,x)\d \meas(x).
\end{split}\]
Hence, by \autoref{lem:covariance} with $f=p_s(\plchldr,y)$ we can compute
\[\begin{split}
\Exp\left[ \int_D (P^*_sr^n)^2 \d\meas \right] &=
\int_D \Exp\left[ \bigl(P^*_sr^n(y)\bigr)^2 \right] \d\meas(y)
= \int_D \Exp\left[
	\left(\int_D p_s(x,y)\d r^n(x)\right)^2 \right] \d\meas(y) \\
&= \int_D \moment{p_s(\plchldr,y)}_2^2 \d\meas(y)
= \int_D \int_D
	\bigl(p_s(x,y)-1\bigr)^2 \d\meas(x) \d\meas(y) \\
&= \int_D \int_D p_s(x,y)^2 \d\meas(x) \d\meas(y) - 1
= \int_D p_{2s}(x,x) \d\meas(x) - 1.
\end{split}\]
By the trace formula \eqref{eq:trace-formula}, \eqref{eq:energy2-n-eps} follows.
\end{proof}

The following lemma basically applies only to $1$-dimensional domains, in view of the ultracontractivity assumption
with $d=1$.

\begin{lemma}[Energy estimate and convergence, $d=1$]\label{lem:energy-estimate}
Let $f^{n,t}$ be the solution to \eqref{eq:energy-estimate}. If $t=t(n)\to0$ as $n\to\infty$, then
\begin{equation}\label{eq:energy-1d-power2}
\lim_{n\to\infty} \Exp\left[\int_D \abs{\nabla f^{n,t}}^2 \d\meas \right] =
\int_0^\infty \left( \int_D p_s(x,x) \d\meas(x) - 1 \right) \d s =
- \sum_{\lambda\in\sigma(\Delta)\setminus\{0\}} \frac1\lambda.
\end{equation}
If ultracontractivity holds with $d=1$ we have also
\begin{equation}\label{eq:energy-1d-power4}
\limsup_{n\to\infty} \Exp\left[\left(
	\int_D \abs{\nabla f^{n,t}}^2 \d\meas \right)^2 \right] < \infty
\end{equation}
and, in particular, the limit in \eqref{eq:energy-1d-power2} is finite.
\end{lemma}

\begin{proof} The identities \eqref{eq:energy-1d-power2} follow by \eqref{eq:energy2-n-eps} by taking the limit as $n\to\infty$.
If ultracontractivity holds with $d=1$, we show that the $\limsup$ in \eqref{eq:energy-1d-power2} is finite by splitting the integration in $(t,1)$ and $(1,\infty)$ in the identity
\begin{equation}\label{eq:ambr4}
\Exp\left[\int_D \abs{\nabla f^{n,t}}^2 \d\meas \right] =
	2 \int_t^\infty \int_D \int_D \bigl(p_s(x,y)-1\bigr)^2 \d\meas(x) \d\meas(y) \d s,
\end{equation}
which is a by-product of the intermediate computations made in the proof of \autoref{lem:rep}. For $s\in(t,1)$ we estimate
\[\begin{split}
\int_D \int_D \bigl(p_s(x,y)-1\bigr)^2 \d\meas(x) \d\meas(y)
&\leq \Cuc s^{-1/2} \int_D \int_D \abs{p_s(x,y)-1} \d\meas(x) \d\meas(y) \\
&\leq 2\Cuc s^{-1/2}.
\end{split}\]
For $s\in(1,\infty)$ instead
\[\begin{split}
\int_D \int_D \bigl(p_s(x,y)-1\bigr)^2 &\d\meas(x) \d\meas(y) =
\int_D \norm{P_{s-1}P_1^*(\delta_y-\meas)}_2^2 \d\meas(y)  \\
&\leq e^{-2\Csg(s-1)} \int_D \norm{P_1^*(\delta_y-\meas)}_2^2 \d\meas(y) \leq
	4\norm{P_1^*}_{{\cal M}\to L^2}^2 e^{-2\Csg(s-1)}.
\end{split}\]
In conclusion, for some geometric constant $C$, one has
\[
\Exp\left[\int_D \abs{\nabla f^{n,t}}^2 \d\meas \right] \leq
C\left( \int_t^1 s^{-1/2} \d s + \int_1^\infty e^{-2\Csg s} \d s \right),
\]
from which the finiteness of \eqref{eq:energy-1d-power2} readily follows.

To show \eqref{eq:energy-1d-power4}, we start from \eqref{eq:energy-1d-representation} 
and estimate with the aid of \autoref{lem:covariance}
\[\begin{split}
\Exp\Biggr[\biggl( \int_D & \abs{\nabla f^{n,t}}^2 \d\meas \biggr)^2 \Biggr] =
\Exp\left[\left( 2 \int_t^\infty \int_D (P_s^*r^n)^2 \d\meas \d s \right)^2\right] \\
&= 4 \int_t^\infty \int_t^\infty \int_D \int_D
	\Exp\left[ \bigl(P^*_sr^n(y)\bigr)^2\bigl(P^*_{s'}r^n(z)\bigr)^2 \right]
	\d\meas(y)\d\meas(z)\d s\d s'  \\
&\begin{split} {}\leq 4 \int_t^\infty \int_t^\infty \int_D \int_D &
	\biggl( 3 \moment{p_s(\plchldr,y)}_2^2\moment{p_{s'}(\plchldr,z)}_2^2  \\
	&+ \frac1n\moment{p_s(\plchldr,y)}_4^2\moment{p_{s'}(\plchldr,z)}_4^2
	\biggr) \d\meas(y)\d\meas(z)\d s\d s' \end{split} \\
&= 3 \left( 2 \int_t^\infty \int_D
	\moment{p_s(\plchldr,y)}_2^2 \d\meas(y)\d s\right)^2
	+ \frac1n \left( 2 \int_t^\infty \int_D
	\moment{p_s(\plchldr,y)}_4^2 \d\meas(y)\d s\right)^2 \\
&= 3 \left(\sum_{\lambda\in\sigma(\Delta)\setminus\{0\}}
	\frac{e^{2t\lambda}}{\lambda}\right)^2
	+ \frac1n \left( 2 \int_t^\infty \int_D
	\moment{p_s(\plchldr,y)}_4^2 \d\meas(y)\d s\right)^2.
\end{split}\]
In order to show that the $\limsup$ of last integral is finite we split the integration in $(t,1)$ and $(1,\infty)$. For $s\in(t,1)$ we use
\[
\moment{p_s(\plchldr,y)}_4^4 \leq \int_D
	\bigl(p_s(x,y)-1\bigr)^4\d\meas(x) \leq \Cuc^3 s^{-3/2} \int_D \abs{p_s(x,y)-1} \d\meas(x) \leq
2 \Cuc^3 s^{-3/2}.
\]
For $s\in(1,\infty)$ instead
\[
\moment{p_s(\plchldr,y)}_4^4 \leq \Cuc^2 s^{-1} \moment{p_s(\plchldr,y)}_2^2 \leq
\Cuc^2 s^{-1} \norm{P_s^*(\delta_y-\meas)}_2^2 \leq
\Cuc^2 e^{-2\Csg(s-1)} \norm{P_1^*(\delta_y-\meas)}_2^2.
\]
Putting these estimates together,
\[
\int_t^\infty \moment{p_s(\plchldr,y)}_4^2 \d s \leq
\sqrt2 \Cuc^{3/2} \int_t^1 s^{-3/4} \d s +
	\Cuc \norm{P^*_1(\delta_y-\meas)}_2 \int_1^\infty e^{-\Csg(s-1)} \d s,
\]
which is bounded, uniformly in $y$ and $t$, because $\norm{P^*_1(\delta_y-\meas)}_2 \leq 2 \norm{P^*_1}_{{\cal M}\to L^2}$.
\end{proof}

\begin{lemma}[Renormalized energy estimate and convergence, $d=2$]\label{lem:renormalized-energy-estimate}
Assume that ultracontractivity holds with $d=2$. Let $f^{n,t}$ be the solution to \eqref{eq:energy-estimate}.
If $t=t(n)\to0$ as $n\to\infty$ and $t\geq C/n$ for some $C>0$, then
\begin{equation}\label{eq:energy-2d-power4}
\limsup_{n\to\infty}
	\frac1{(\log t)^2} \Exp\left[\int_D \abs{\nabla f^{n,t}}^4 \d\meas \right] < \infty.
\end{equation}
In particular
\begin{equation}\label{eq:energy-2d-power2-limsup}
\limsup_{n\to\infty}
	\frac1{\abs{\log t}}\Exp\left[\int_D \abs{\nabla f^{n,t}}^2 \d\meas \right]  < \infty.
\end{equation}
Moreover, under the assumptions on $D$ of \autoref{prop:expansion}, one has
\begin{equation}\label{eq:energy-2d-convergence}
\lim_{n\to\infty}
	\frac{1}{\abs{\log t}}\Exp\left[\int_D \abs{\nabla f^{n,t}}^2 \d\meas \right]  =\frac{1}{4\pi}.
\end{equation}
\end{lemma}

\begin{proof} We will prove first \eqref{eq:energy-2d-power2-limsup} as an intermediate step in the proof of
\eqref{eq:energy-2d-power4}, starting from the representation formula \eqref{eq:ambr4}.
For $s\in(t,1)$ we estimate
\[
\int_D \int_D \bigl(p_s(x,y)-1\bigr)^2 \d\meas(x) \d\meas(y) \leq
\Cuc s^{-1} \int_D \int_D \abs{p_s(x,y)-1} \d\meas(x) \d\meas(y) \leq 2\Cuc s^{-1}.
\]
For $s\in(1,\infty)$ instead
\[\begin{split}
\int_D \int_D \bigl(p_s(x,y)-1\bigr)^2 &\d\meas(x) \d\meas(y) =
\int_D \norm{P_{s-1}P_1(\delta_y-\meas)}_2^2 \d\meas(y) \\
&\leq e^{-2\Csg(s-1)} \int_D \norm{P_1(\delta_y-\meas)}_2^2 \d\meas(y) \leq
	4\norm{P_1}_{L^1\to L^2}^2 e^{-2\Csg(s-1)}.
\end{split}\]
In conclusion, for some geometric constant $C$, one has
\[
\Exp\left[\int_D \abs{\nabla f^{n,t}}^2 \d\meas \right] \leq
C\left( \int_t^1 s^{-1} \d s + \int_1^\infty e^{-2\Csg s} \d s \right) \leq
C (\abs{\log t} + 1),
\]
from which \eqref{eq:energy-2d-power2-limsup} readily follows.

In order to prove \eqref{eq:energy-2d-convergence}, we notice that the estimates given in the proof of
\eqref{eq:energy-2d-power2-limsup} show that
\[
\frac{1}{\abs{\log t}}\Exp\left[\int_D \abs{\nabla f^{n,t}}^2 \d\meas \right] 
-\frac{2}{\abs{\log t}}
\int_t^1\left(\int_D p_{2s}(x,x) \d\meas(x) - 1\right)\d s
\]
is infinitesimal as $n\to\infty$. Combining this information with \eqref{eq:expansion} of
\autoref{prop:expansion}, we obtain \eqref{eq:energy-2d-convergence}.

To deal with \eqref{eq:energy-2d-power4}, we introduce the Paley-Littlewood function
\[
S(g) = \left( \int_0^\infty \left(s \partial_s P_sg \right)^2 \frac{\d s}{s} \right)^{1/2}.
\]
Using the Riesz transform bound and the fundamental theorem \cite{S70} 
$\norm{g}_p^p \leq c_p \norm{S(g)}_p^p$ for any $p\in(1,\infty)$ and $g$ with $\int_D g\d\meas=0$, we obtain
\[\begin{split}
\int_D \abs{\nabla f^{n,t}}^4 &\d\meas \leq
	C_r \int_D \abs*{(-\Delta)^{1/2}f^{n,t}}^4 \d\meas \leq
C_r c_4 \int_D S\left((-\Delta)^{1/2}f^{n,t}\right)^4 \d\meas \\
&= C_r c_4 \int_0^\infty \int_0^\infty \int_D
	s \left(\partial_s P_s (-\Delta)^{1/2}f^{n,t}\right)^2
	s' \left(\partial_{s'} P_{s'} (-\Delta)^{1/2}f^{n,t}\right)^2
	\d\meas \d s \d s'.
\end{split}\]
Using the fact that $\partial_tP_t = \Delta P_t$ and that the operators $\Delta$, $P_t$ and $(-\Delta)^{1/2}$ commute we have
\[
\partial_\tau P_\tau (-\Delta)^{1/2}f^{n,t} = (-\Delta)^{1/2}P^*_{\tau+t}r^n,
\]
so that
\[
\int_D \abs{\nabla f^{n,t}}^4 \d\meas \leq
C_r c_4 \int_0^\infty \int_0^\infty \int_D
	\left( (-s\Delta)^{1/2} P^*_{s+t} r^n \right)^2
	\left( (-s'\Delta)^{1/2} P^*_{s'+t} r^n \right)^2 \d\meas \d s \d s'.
\]
For $y\in D$ fixed, consider the operators
\[
(T_s^t \mu)(y) = \left( (-s\Delta)^{1/2} P^*_{s+t} \mu \right)(y) =
\int_D K_s^t(x,y) \d\mu(x)
\]
and notice that
\[
 \int_D K_s^t(x,y) \d\meas(x)= T_s^t\meas(y)= 0.
\]
In addition, since $T_s^t: L^2(\meas) \to L^2(\meas)$ is self-adjoint, the kernel $K_s^t$ is symmetric and
\begin{equation}\label{eq:ambr10}
T^t_s\delta_x(y)=K^t_s(x,y)=K^t_s(y,x).
\end{equation}
Taking the expectation of the integrand,
\[\begin{split}
\Exp\biggl[
	\left( (-s\Delta)^{1/2} P^*_{s+t} r^n \right)^2&(y)
	\left( (-s'\Delta)^{1/2} P^*_{s'+t} r^n \right)^2(y) \biggr] =
\Exp\left[ (T_s^t r^n)^2(y) (T_{s'}^t r^n)^2(y) \right] \\
&= \Exp\left[
	\left(\int_D K_s^t(x,y) \d r^n(x) \right)^2
	\left(\int_D K_{s'}^t(x',y) \d r^n(x') \right)^2 \right] \\
&\leq 3\frac{n-1}{n} \moment{K_s^t(\plchldr,y)}_2^2 \moment{K_{s'}^t(\plchldr,y)}_2^2
	+ \frac1n \moment{K_s^t(\plchldr,y)}_4^2 \moment{K_{s'}^t(\plchldr,y)}_4^2.
\end{split}\]
Integrating in $s$ and $s'$ we obtain
\begin{multline*}
\int_0^\infty \int_0^\infty \Exp\biggl[
	\left( (-s\Delta)^{1/2} P^*_{s+t} r^n \right)^2(y)
 \left( (-s'\Delta)^{1/2} P^*_{s'+t} r^n \right)^2(y) \biggr] \d s \d s' \\
\leq 3\frac{n-1}{n} \left( \int_0^\infty \moment{K_s^t(\plchldr,y)}_2^2 \d s \right)^2
	+ \frac1n \left( \int_0^\infty \moment{K_{s'}^t(\plchldr,y)}_4^2 \d s' \right)^2.
\end{multline*}
Since $(P_t)_{t\ge 0}$ is a bounded analytic semigroup, complex interpolation yields  that, for $p \in (1, \infty)$, $(-\tau\Delta)^{1/2}P_{\tau/2}:L^p\to L^p$ is continuous with norms uniformly bounded for $\tau\geq0$ \cite[Sections X.10-11]{Y80}, hence
we have the estimate
\[\begin{split}
\moment{K_s^t(\plchldr,y)}_p &=
\norm{T_s^t \delta_y}_p = \norm{T_s^t (\delta_y-\meas)}_p \\
&= \norm*{ (-s\Delta)^{1/2}P_{s/2} P^*_{s/2+t} (\delta_y-\meas) }_p \\
&\leq C_p \norm{P^*_{s/2+t} (\delta_y-\meas)}_p,
\end{split}\]
where in the first equality we used \eqref{eq:ambr10}. We consider 
\[
\int_0^\infty \moment{K_s^t(\plchldr,y)}_p^2 \d s\leq
C_p \int_0^\infty \norm{P^*_{s/2+t} (\delta_y-\meas)}_p^2 \d s =
2C_p \int_t^\infty \norm{P^*_s (\delta_y-\meas)}_p^2 \d s.
\]
Now we split the integrals for $s\in(t,2)$ and $s\in(2,\infty)$.
In the former interval we use the estimate
\[\begin{split}
\norm{P^*_s(\delta_y-\meas)}_p &=
\left( \int_D \abs{p_s(x,y)-1}^p \d\meas(x) \right)^{1/p} \\
&\leq \left( \int_D \bigl(\Cuc s^{-1}\bigr)^{p-1} \abs{p_s(x,y)-1} \d\meas (x) \right)^{1/p} \\
&\leq 2^{1/p} \Cuc^{(p-1)/p} s^{-(p-1)/p}.
\end{split}\]
In the latter interval we use the estimate
\[\begin{split}
\norm{P_s^*(\delta_y-\meas)}_p &= \norm{P_1P_{(t-2)}P^*_1(\delta_y-\meas)}_p \\
&\leq \norm{P_1}_{L^2\to L^p} e^{-\Csg(s-2)} \norm{P^*_1(\delta_y-\meas)}_2  \\
&\leq \norm{P_1}_{L^2\to L^p} \norm{P^*_1}_{{\cal M}\to L^2}
	e^{-\Csg(s-2)} \norm{\delta_y-\meas}_{{\cal M}} \\
&\leq 2e^2 \norm{P_1}_{L^2\to L^p} \norm{P^*_1}_{{\cal M}\to L^2} e^{-\Csg t}.
\end{split}\]
Putting these estimates together, in the case $p=2$ we have,
\[
\int_t^\infty \norm{P^*_s (\delta_y-\meas)}_p^2 \d s \leq
C \left( \int_t^2 s^{-1} \d s + \int_2^\infty e^{-2\Csg s} \d s \right) \leq
C (\abs{\log t} + 1)
\]
for some geometric constant $C$. In the case $p=4$ we have also
\[
\int_t^\infty \norm{P^*_s (\delta_y-\meas)}_p^2 \d s \leq
C \left( \int_t^2 s^{-3/2} \d s + \int_2^\infty e^{-2\Csg s} \d s \right) \leq
C \left(\frac 1{\sqrt{t}} + 1\right).
\]
This yields
\[\begin{split}
3\frac{n-1}{n} \left( \int_0^\infty \moment{K_s^t(\plchldr,y)}_2^2 \d s \right)^2
	+ \frac1n \left( \int_0^\infty \moment{K_s^t(\plchldr,y)}_4^2 \d s \right)^2  \\
\leq C \frac{n-1}n (\log t)^2 + C\left(\frac n t + 1\right).
\end{split}\]
In conclusion
\[\begin{split}
\Exp\left[\int_D \abs{\nabla f^{n,t}}^4 \d\meas \right] \frac1{(\log t)^2} &\leq
\frac{C}{(\log t)^2}
	\int_D \left(\frac{n-1}n (\log t)^2 + \frac n t + 1\right) \d\meas(y) \\
&\leq C \left[ 1 + \frac n{(\log t)^2t} + \frac1{(\log t)^2} \right]
\end{split}\]
is uniformly bounded as $n\to\infty$ by the assumptions on $t=t(n)$.
\end{proof}

\section{Proof of the main result}\label{sec:proofmain}

In this section we prove \autoref{thm:main}. In the proof of the upper bound we need only
to assume the regularizing properties of $P_t$ listed in \autoref{sec:heat}; in particular this inequality
covers also the case $D=[0,1]^2$ and compact 2-dimensional Riemannian manifolds with smooth boundary. 
In the proof of the lower bound we need also to assume that $D$ has no boundary; by a comparison argument,
since the distance in $\T^2$ is smaller than the distance in $[0,1]^2$, we recover also the lower bound for $D=[0,1]^2$.

We include also the $1$-dimensional case (whose proofs are a bit simpler), which covers the case of the interval and the case of the circle.
For brevity we state the result only in the Riemannian case, but the strength of this method relies in the fact that it can be extended to more general $1$-dimensional spaces (see also \autoref{sec:more_general_spaces}).

\begin{theorem}\label{thm:main1}
Assume that either $D=[0,1]$ or $D=\T^1$. Then
\[
\lim_{n\to\infty} n\Exp\left[W_2^2(\mu^n,\meas)\right] =
- \sum_{\lambda\in\sigma(\Delta)\setminus\{0\}} \frac1\lambda.
\]
In particular, from Euler's formula $\pi^2=6\sum_{k\geq 1}k^{-2}$, the limit equals $1/6$ for $D=[0,1]$ and $1/12$ for $D=\T^1$. 
\end{theorem}

\begin{remark}\label{rem:1deasy}
In the case $D=[0,1]$ and $\meas=\leb^1\res D$ we can explicitly compute $n\Exp[W_2^2(\mu^n,\meas)]$ and $n\Exp[W_2^2(\mu^n,\nu^n)]$ as follows (and in particular, the former is identically equal to $1/6$). For any fixed $n\in\setN$, let $X_{(k)}$ and $Y_{(k)}$ denote the order statistics of the random variables $(X_i)_{i=1}^n$ and $(Y_i)_{i=1}^n$. It is well known that $X_{(k)}$ and $Y_{(k)}$ are distributed according to the beta distribution
$X_{(k)} \sim Y_{(k)} \sim \Beta(k, n+1-k)$.

The optimal map is given by the monotone rearrangement of the mass, therefore
\[
\begin{split}
\Exp\left[W_2^2(\mu^n,\meas)\right]
&= \Exp\left[ \sum_{k=1}^n \int_{(k-1)/n}^{k/n} (X_{(k)}-t)^2 \d t \right]
= \sum_{k=1}^n \int_{(k-1)/n}^{k/n} \Exp[(X_{(k)}-t)^2] \d t \\
&= \sum_{k=1}^n \int_{(k-1)/n}^{k/n} \left(
	\Var(X_{(k)}) + (\Exp[X_{(k)}]-t)^2	\right) \d t \\
&= \sum_{k=1}^n \int_{(k-1)/n}^{k/n} \left[
	\frac{(k+1)k}{(n+2)(n+1)} - 2t\frac{k}{n+1} + t^2 \right] \d t\\
&= \sum_{k=1}^n \left[
	\frac{(k+1)k}{(n+2)(n+1)n} - \frac{(2k-1)k}{(n+1)n^2}
	+ \frac{3k^2-3k+1}{3n^3} \right]
= \frac{1}{6n}.
\end{split}
\]

Similarly, in the bipartite case we have
\[
\begin{split}
\Exp\left[W_2^2(\mu^n,\nu^n)\right] &=
\Exp\left[ \frac1n \sum_{k=1}^n (X_{(k)}-Y_{(k)})^2 \right]
= \frac1n \sum_{k=1}^n \left( \Exp[X_{(k)}^2] - \Exp[X_{(k)}^2]\right) \\
&= \frac2n \sum_{k=1}^n \Var(X_{(k)})
= \frac2n \sum_{k=1}^n \frac{k(n+1-k)}{(n+1)^2(n+2)} = \frac1{3(n+1)}.
\end{split}
\]
\end{remark}

\subsection{Upper bound}

\begin{theorem}[Upper bound, $d=1$]
Assume that ultracontractivity holds with $d=1$. 
Then
\[
\limsup_{n\to\infty} n \Exp\left[W_2^2(\mu^n,\meas)\right]  \leq
- \sum_{\lambda\in\sigma(\Delta)\setminus\{0\}} \frac1\lambda.
\]
\end{theorem}

\begin{proof}
Fix $q\in(1/2,1)$, $\eta\in (0,1)$ and let $t=t(n)=n^{-2q}$. For $\eta\in(0,1)$ consider the event
\[
A_\eta = A_{\eta,n} = \left\{ \sup_{y\in D} \frac{\abs{r^{n,t}(y)}}{\sqrt n} \leq \eta \right\}.
\]
By \autoref{prop:d1-scaling}, since $W_2^2(\mu^n,\meas)\leq(\diam D)^2$, for $n$ large enough we have
\[\begin{split}
n\Exp\left[W_2^2(\mu^n,\meas)\right] &=
n\Exp\left[W_2^2(\mu^n,\meas)\chi_{A_\eta}\right] +
	n\Exp\left[W_2^2(\mu^n,\meas)\chi_{A_\eta^c}\right]  \\
&\leq n\Exp\left[W_2^2(\mu^n,\meas)\chi_{A_\eta}\right] +
	C (\diam D)^2 n \exp\left(-\gamma n^{1-q}\right)
\end{split}\]
with $C=C(C_D,\Cg)$ and $\gamma=\gamma(\eta,\Cuc)>0$.

Using the Young inequality for products with $\alpha>0$ and $W_2^2(\mu^n,\mu^{n,t})\leq\Cdr t$ we have
\[\begin{split}
W_2^2(\mu^n,\meas) &\leq
\bigl(W_2(\mu^{n,t},\meas) + W_2(\mu^n,\mu^{n,t})\bigr)^2 \\
&\leq (1+\alpha)W_2^2(\mu^{n,t},\meas) + (1+\alpha^{-1})W_2^2(\mu^n,\mu^{n,t}) \\
&\leq (1+\alpha)W_2^2(\mu^{n,t},\meas) + (1+\alpha^{-1})\Cdr t.
\end{split}\]
Therefore, since $n t\to0$, it is sufficient to estimate
\[
\limsup_{n\to\infty} n\Exp\left[W_2^2(\mu^{n,t},\meas)\chi_{A_\eta}\right].
\]
To this end, we apply \autoref{prop:dacorogna-moser} with $u_0=u^{n,t}$ and $u_1=1$.
Since $f^{n,t}$ solves \eqref{eq:energy-estimate} from \autoref{prop:dacorogna-moser} we get
\[
W_2^2(\mu^{n,t},\meas) \leq
\frac1n \int_D \frac{\abs{\nabla f^{n,t}}^2}{M(u^{n,t},1)} \d\meas.
\]
In the event $A_\eta$ we have $u^{n,t}\geq1-\eta$ in $D$, hence the first inequality in \eqref{eq:ambr2} gives
\[
\frac1{M(u^{n,t},1)} \leq \frac1{\sqrt{u^{n,t}}} \leq \frac1{\sqrt{1-\eta}}.
\]
The previous two inequalities and \autoref{lem:energy-estimate} 
give
\[\begin{split}
\limsup_{n\to\infty} n \Exp\left[W_2^2(\mu^{n,t},\meas)\chi_{A_\eta}\right]
&\leq \lim_{n\to\infty}
	\frac1{\sqrt{1-\eta}} \Exp\left[ \int_D \abs{\nabla f^{n,t}}^2 \d\meas \right] \\
&= - \frac1{\sqrt{1-\eta}} \sum_{\lambda\in\sigma(\Delta)\setminus\{0\}} \frac1\lambda.
\end{split}\]
In conclusion we have
\[
\limsup_{n\to\infty} n\Exp\left[W_2^2(\mu^n,\meas)\right] \leq
- \frac{1+\alpha}{\sqrt{1-\eta}} \sum_{\lambda\in\sigma(\Delta)\setminus\{0\}} \frac1\lambda
\]
and we obtain the thesis by letting first $\alpha\to0$ and then $\eta\to0$.
\end{proof}

\begin{theorem}[Upper bound, $d=2$]
Assume that ultracontractivity holds with $d=2$ 
and that $D$ is as in \autoref{prop:expansion}. Then
\[
\limsup_{n\to\infty} \frac{n}{\log n} \Exp\left[W_2^2(\mu^n,\meas)\right]
\leq \frac 1{4\pi}.
\]
\end{theorem}

\begin{proof}
Fix $\gamma>0$ and let $t(n)=c(n)=\gamma n^{-1}\log n$. Let us set
\[
\mu^{n,t,c} = (1-c)\mu^{n,t}+c\meas, \qquad
u^{n,t,c} = (1-c)u^{n,t}+c
\]
as in \autoref{lem:log-mean-convergence}. From the joint convexity of $W_2^2$ (see \eqref{eq:joint_convexity}) we immediately get
\[
W_2^2(\mu^{n,t},\mu^{n,t,c}) \leq (\diam D)^2 c
\]
Using the Young inequality for products with $\alpha>0$ and $W_2^2(\mu^n,\mu^{n,t})\leq\Cdr t$, we have
\[\begin{split}
W_2^2(\mu^n,\meas) &\leq
\bigl(W_2(\mu^{n,t,c},\meas) + W_2(\mu^n,\mu^{n,t})
	+ W_2(\mu^{n,\eps},\mu^{n,t,c})\bigr)^2 \\
&\leq (1+\alpha)W_2^2(\mu^{n,t,c},\meas) + 2(1+\alpha^{-1})
	\bigl[W_2^2(\mu^n,\mu^{n,t}) + W_2^2(\mu^{n,t},\mu^{n,t,c})\bigr] \\
&\leq (1+\alpha)W_2^2(\mu^{n,t,c},\meas) + 2(1+\alpha^{-1})
	[\Cdr t + (\diam D)^2c].
\end{split}\]
We start by estimating the contribution of the first term.

To this end we apply \autoref{prop:dacorogna-moser} with $u_0=u^{n,t,c}$ and $u_1=1$.
Recalling that $f^{n,t}$ solves the PDE $\Delta f^{n,t}=\sqrt{n}(u^{n,t}-1)$ with homogeneous 
Neumann boundary conditions, we get
\[
\Delta \frac{(1-c)}{\sqrt{n}}f^{n,t}=(u^{n,t,c}-1),
\]
hence \autoref{prop:dacorogna-moser} gives
\begin{equation}\label{eq:dac-mos-2d}
W_2^2(\mu^{n,t,c},\meas) \leq
\frac{(1-c)^2}n \int_D \frac{\abs{\nabla f^{n,t}}^2}{M(u^{n,t,c},1)} \d\meas \leq
\frac1n \int_D \frac{\abs{\nabla f^{n,t}}^2}{M(u^{n,t,c},1)} \d\meas.
\end{equation}
Adding and subtracting $\abs{\nabla f^{n,t}}^2$ to the integrand we obtain
\[
\Exp\left[ \int_D \frac{\abs{\nabla f^{n,t}}^2}{M(u^{n,t,c},1)} \d\meas \right] =
\Exp\left[ \int_D \abs{\nabla f^{n,t}}^2 \d\meas \right] +
	\Exp\left[ \int_D \abs{\nabla f^{n,t}}^2
		\left( M(u^{n,t,c},1)^{-1}-1 \right) \d\meas \right]
\]
We deal with the two addends separately. For the former, since the function $f^{n,t}$ solves \eqref{eq:energy-estimate}, $t\geq C/n$ and
$\abs{\log t}/\log n \to 1$ as $n\to\infty$, 
\autoref{lem:renormalized-energy-estimate} gives
\[
\lim_{n\to\infty}\frac1{\log n} \Exp\left[ \int_D \abs{\nabla f^{n,t}}^2 \d\meas \right]
	 =\frac 1{4\pi}. 
\]
For the latter, by Cauchy-Schwarz inequality we have
\[\begin{split}
&\limsup_{n\to\infty} \frac{1}{\log n}\Exp\left[ \int_D \abs{\nabla f^{n,t}}^2
		\left( M(u^{n,t,c},1)^{-1}-1 \right) \d\meas \right]  \leq \\
&\quad\leq \left( \limsup_{n\to\infty}\frac{1}{\log n}
	\Exp\left[ \int_D \abs{\nabla f^{n,t}}^4 \d\meas \right]^{1/2}  \right)
\left(\lim_{n\to\infty}
	\Exp\left[ \int_D \left( M(u^{n,t,c},1)^{-1}-1 \right)^2 \d\meas \right]^{1/2}
	\right)
\end{split}\]
which converges to $0$ by \autoref{lem:renormalized-energy-estimate} and \autoref{lem:log-mean-convergence}.

Recalling \eqref{eq:dac-mos-2d}, we deduce
\[
\limsup_{n\to\infty} \frac{n}{\log n}\Exp [ W_2^2(\mu^{n,t,c},\meas) ]  \leq
\frac 1{4\pi}. 
\]
In conclusion
\[
\limsup_{n\to\infty} \frac{n}{\log n} \Exp\left[W_2^2(\mu^n,\meas)\right]
\leq \frac{(1+\alpha) }{4\pi}
	+ 2(1+\alpha^{-1}) [\Cdr+(\diam D)^2] \gamma
\]
and the thesis follows letting first $\gamma\to0$ and then $\alpha\to0$.
\end{proof}

\subsection{Lower bound}

\begin{theorem}[Lower bound, $d=1$]
Assume that ultracontractivity holds with $d=1$ and that $N_D(\delta)\leq \max\{1,C_D\delta^{-1}\}$ for every $\delta>0$. Then
\[
\liminf_{n\to\infty} n\Exp\left[W_2^2(\mu^n,\meas)\right] \geq
- \sum_{\lambda\in\sigma(\Delta)\setminus\{0\}} \frac1\lambda.
\]
\end{theorem}

\begin{proof}
Fix $q\in(0,1)$, $\eta\in (0,1)$ and let $t=t(n)=\eta^{-1} n^{-2q}$. By
\autoref{prop:d1-scaling} the complement of the event
\begin{equation}\label{eq:ambr15}
A_\eta = A_{\eta,n} =
\left\{ \sup_{y\in D} \frac{\abs{r^{n,t}(y)}}{\sqrt n} \leq \eta \right\}
\end{equation}
has infinitesimal probability as $n\to\infty$.
By the contractivity assumption we have
\[
W_2^2(\mu^n,\meas) \geq e^{2Kt}W_2^2(\mu^{n,t},\meas).
\]
Therefore it is sufficient to estimate $\liminf\limits_{n\to\infty} n\Exp[W_2^2(\mu^{n,t},\meas)\chi_{A_\eta}]$ from below.
By duality, 
\begin{equation}\label{eq:ambr16}\begin{split}
\frac12 W_2^2(\mu^{n,t},\meas) &\geq
	\sup\left\{ \int_D f\d\mu^{n,t} + \int_D g\d\meas
	\ \bigg\vert\ f(x)+g(y) \leq \frac{\dist(x,y)^2}{2} \right\} \\
&= \sup\left\{ \int_D f \frac{\d r^{n,t}}{\sqrt n}+\int_D (f+g)\d\meas  
	\ \bigg\vert\ f(x)+g(y) \leq \frac{\dist(x,y)^2}{2} \right\}.
\end{split}\end{equation}
Let $f^{n,t}$ be the solution to \eqref{eq:energy-estimate} and define $f=-f^{n,t}/\sqrt{n}$, so that $\norm{\Delta f}_\infty\leq\eta$ in the event $A_\eta$, and we can estimate thanks to \eqref{eq:elliptic-regularity}
\begin{equation}\label{eq:ambr18}
\norm{(\Delta f)^+}_\infty+\frac{e^{2K^-t}-1}{2}\norm{\nabla f}_\infty^2\leq \omega(\eta)
\end{equation}
with $\omega(\eta)\to 0$ as $\eta\to 0$.
 To this function $f$ we associate the potential $g$ given by \autoref{cor:good-enough-potential}, hence we get (still
in the event $A_\eta$)
\[\begin{split}
\frac12 W_2^2(\mu^{n,\eps},\meas) &\geq
	\int_D (f+g)\d\meas + \int_D f \frac{\d r^{n,t}}{\sqrt n} \\
&\geq - e^{\omega(\eta)} \int_D \frac{\abs{\nabla f}^2}2 \d\meas
	- \int_D f \Delta f \d\meas = \\
&= \left(1-\frac{e^{\omega(\eta)}}2\right)
	\frac1n \int_D \abs{\nabla f^{n,t}}^2 \d\meas.
\end{split}\]
Thus, by \autoref{lem:energy-estimate},
\[\begin{split}
\frac{1}{2-e^{\omega(\eta)}}\liminf_{n\to\infty}
	&n\Exp\left[W_2^2(\mu^{n,t},\meas)\chi_{A_\eta}\right] \geq
\liminf_{n\to\infty} \Exp\left[
	\chi_{A_\eta} \int_D \abs{\nabla f^{n,t}}^2 \d\meas \right] \\
&\geq \lim_{n\to\infty} \Exp\left[
	\int_D \abs{\nabla f^{n,t}}^2 \d\meas \right]
	- \limsup_{n\to\infty} \Exp\left[
	\chi_{A^c_\eta} \int_D \abs{\nabla f^{n,t}}^2 \d\meas \right]  \\
&= - \sum_{\lambda\in\sigma(\Delta)\setminus\{0\}} \frac1\lambda
\end{split}\]
because
\[
\limsup_{n\to\infty} \Exp\left[
	\chi_{A_\eta^c} \int_D \abs{\nabla f^{n,t}}^2 \d\meas \right] \leq
\limsup_{n\to\infty} \Pro(A_{\eta,n}^c)^{1/2}
\left(\Exp\left[\left(
	\int_D \abs{\nabla f^{n,t}}^2 \d\meas \right)^2\right]\right)^{1/2}
= 0
\]
by H\"older inequality and \eqref{eq:energy-1d-power4}.
The thesis follows letting $\eta\to0$.
\end{proof}

\begin{theorem}[Lower bound, $d=2$]\label{thm:lwb2}
Assume that ultracontractivity holds with $d=2$ and
that $N_D(\delta)\leq C\delta^{-2}$ for every $\delta>0$. Then
\[
\liminf_{n\to\infty} \Exp\left[W_2^2(\mu^n,\meas)\right] \frac{n}{\log n} \geq\frac{1}{4\pi}.
\]
\end{theorem}

\begin{proof}
By \autoref{prop:d2-scaling}, for any $\eta\in(0,1)$ there is $\gamma>0$ such that, if we let $t=t(n)=\gamma n^{-1}\log n$, the event 
$A_\eta$ in \eqref{eq:ambr15} satisfies $\Pro(A_\eta^c) \leq C/n$, for $n$ large enough and some $C>0$ independent of $n$.
As in the previous proof, thanks to contractivity it is sufficient to estimate from below
\[
\liminf_{n\to\infty} \frac{n}{\log n}
	\Exp\left[W_2^2(\mu^{n,t},\meas)\chi_{A_\eta}\right] .
\]
Let $f^{n,t}$ be the solution to \eqref{eq:energy-estimate} and define $f=-f^{n,t}/\sqrt{n}$, so that $\norm{\Delta f}_\infty\leq\eta$ in the event $A_\eta$. To this function $f$ we associate the potential $g$ given by \autoref{cor:good-enough-potential}, hence 
thanks to the duality formula \eqref{eq:ambr16} we can estimate (in the event $A_\eta$)
\[\begin{split}
\frac12 W_2^2(\mu^{n,t},\meas) &\geq
	\int_D (f+g)\d\meas + \int_D f \frac{\d r^{n,t}}{\sqrt n} \\
&\geq - e^{\omega(\eta)} \int_D \frac{\abs{\nabla f}^2}2 \d\meas
	- \int_D f \Delta f \d\meas \\
&= \left(1-\frac{e^{\omega(\eta)}}2\right)
	\frac1n \int_D \abs{\nabla f^{n,t}}^2 \d\meas
\end{split}\]
with $\omega(\eta)$ as in \eqref{eq:ambr18}.
Since $t\geq C/n$ for some positive constant $C$ and $\abs{\log t}/\log n\to 1$, from
\autoref{lem:renormalized-energy-estimate} we get
\[\begin{split}
\frac{1}{2-e^{\omega(\eta)}} &\liminf_{n\to\infty} \frac{n}{\log n}\Exp\left[W_2^2(\mu^{n,t},\meas)\chi_{A_\eta}\right]  \geq
\liminf_{n\to\infty} \Exp\left[
	\chi_A \int_D \abs{\nabla f^{n,t}}^2 \d\meas \right]
	\frac{1}{\log n} \\
&\geq \lim_{n\to\infty} \Exp\left[
	\int_D \abs{\nabla f^{n,t}}^2 \d\meas \right]\frac{1}{\log n}
	- \limsup_{n\to\infty} \Exp\left[
	\chi_{A^c} \int_D \abs{\nabla f^{n,t}}^2 \d\meas \right]
	\frac{1}{\log n} \\
&= \frac 1{4\pi}
\end{split}\]
because
\begin{multline*}
\limsup_{n\to\infty} \Exp\left[
	\chi_{A_\eta^c} \int_D \abs{\nabla f^{n,t}}^2 \d\meas \right] \frac{1}{\log n}\leq \\
\leq \limsup_{n\to\infty} \Pro(A_\eta^c)^{1/2}
\left(\Exp\left[ \int_D \abs{\nabla f^{n,t}}^4 \d\meas \right] \frac{1}{(\log n)^2}\right)^{1/2}
= 0
\end{multline*}
by H\"older inequality and \eqref{eq:energy-2d-power4}.
In conclusion we have
\[
\liminf_{n\to\infty}  \frac{n}{\log n} \Exp\left[W_2^2(\mu^n,\meas)\right]
\geq \frac 1{4\pi}
\]
and the thesis follows letting $\eta\to0$.
\end{proof}

\subsection{The bipartite case}

We prove now the bipartite part of \autoref{thm:main}. 
It will be convenient to introduce a notation $(\Omega,\Pro)$ for the underlying probability space. 

\begin{lemma} Let $D\subset\setR^d$ be a compact set and assume that $\meas\in\Prob(D)$ is absolutely continuous w.r.t.\ the Lebesgue measure. The $L^2(D,\meas;D)$-valued maps 
\[
\Omega\ni\omega\mapsto T^{\mu^n(\omega)},\quad
\Omega\ni\omega\mapsto T^{\nu^n(\omega)}
\]
providing the optimal maps from $\meas$ to $\mu^n(\omega)$ and $\nu^n(\omega)$ are measurable and independent. 
\end{lemma}

\begin{proof} The independence of $(X_i,Y_i)$ easily implies that the two measure-valued random variables
 $\mu^n(\omega)$, $\nu^n(\omega)$ are measurable  and independent, where in $\Prob(D)$ we consider the Borel
 $\sigma$-algebra induced by the topology of weak convergence in duality with $C(D)$. Now, recalling
 \autoref{optimaT}, since independence is stable under composition with continuous functions the statement follows.
\end{proof}

\begin{proposition}
Let $D\subset\setR^d$ be a bounded domain. For all $n\geq 1$ one has
\begin{equation}\label{eq:ambr23}
\Exp\left[W_2^2(\mu^n,\nu^n)\right] =2 \Exp\left[W_2^2(\mu^n,\meas)\right] -
 2 \int_D \abs*{\Exp\left[T^{\mu^n}(x)-x\right]}^2 \d\meas(x) .
\end{equation}
\end{proposition}

\begin{proof}
If $S,T:\Omega\to L^2(D,\meas;\setR^d)$ are independent, one has the identity
\begin{equation}\label{eq:ambr17}
\Exp\left[\int_D \langle S^\omega(x), T^\omega(x)\rangle \d\meas(x) \right]=
\int_D\langle\Exp\left[S^\omega(x)\right],\Exp\left[T^\omega(x)\right]\rangle\d\meas(x).
\end{equation}
We sketch the argument of the proof: if $S=\lambda e$, $T=\lambda' e'$, with $\lambda,\lambda':\Omega\to\setR^d$ and
$e,e'$ orthogonal unit vectors of $L^2(D,\meas)$,
then $\lambda$ and $\lambda'$ are independent and \eqref{eq:ambr17} reduces
to $\Exp[\langle\lambda,\lambda'\rangle]=\langle\Exp[\lambda],\Exp[\lambda']\rangle$. By bilinearity, \eqref{eq:ambr17} still holds if $S$ and $T$ take
their values in the vector space generated on $\setR^d$ by a finite orthonormal set $\{e_1,\ldots,e_k\}$ of $L^2(D,\meas)$.
By a standard projection argument, and by approximation, we recover the general result.

 For all $\omega\in\Omega$ the plan $(T^{\mu^n(\omega)},T^{\nu^n(\omega)})_\#\meas$ is a coupling between $\mu^n(\omega)$ and 
$\nu^n(\omega)$. Hence (omitting for simplicity the dependence on $\omega$) and using \eqref{eq:ambr17} with
$S=T^{\mu_n}$, $T=T^{\nu^n}$ one has
\[
\begin{split}
\Exp&\left[W_2^2(\mu^n,\nu^n)\right] \leq \Exp\left[\int_D \abs*{T^{\mu^n}-T^{\nu^n}}^2 \d\meas\right] \\
&= \Exp\left[\int_D \left(\abs*{T^{\mu^n}(x)-x}^2 + \abs*{T^{\nu^n}(x)-x}^2
- 2 \langle T^{\mu^n}(x)-x, T^{\nu^n}(x)-x \rangle \right)\d\meas(x)\right] \\
&= 2 \Exp\left[W_2^2(\mu^n,\meas)\right]
- 2 \Exp\left[\int_D \langle T^{\mu^n}(x)-x, T^{\nu^n}(x)-x \rangle \d\meas(x) \right]  \\
&= 2 \Exp\left[W_2^2(\mu^n,\meas)\right] - 2 \int_D \abs*{\Exp\left[T^{\mu^n}(x)-x\right]}^2 \d\meas(x),
\end{split}
\]
where we used that $\Exp\left[W_2^2(\mu^n,\meas)\right] = \Exp\left[W_2^2(\nu^n,\meas)\right]$ since $\mu^n$ and $\nu^n$ have the same law.
\end{proof}

In particular, combining the inequality in \eqref{eq:ambr23} (neglecting for a moment the
negative term in the right hand side) with the first part of \autoref{thm:main}, we obtain
\begin{equation}\label{eq:ambr21}
\limsup_{n\to\infty} \frac{n}{\log n}\Exp\left[W_{2,\T^2}^2(\mu^n,\nu^n)\right]\leq
\limsup_{n\to\infty} \frac{n}{\log n}\Exp\left[W_{2,[0,1]^2}^2(\mu^n,\nu^n)\right]\leq \frac{1}{2\pi}.
\end{equation}

Next, we deal with lower bounds. It will be sufficient, by a comparison argument, to provide the lower bound
only in the flat torus.

\begin{proposition}
Let $\D=\T^2$. Then
\begin{equation}
\liminf_{n\to\infty}\frac{n}{\log n}\Exp\left[W_2^2(\mu^n,\nu^n)\right] \geq \frac 1{2\pi}.
\end{equation}
\end{proposition}

\begin{proof}
Similarly to the proof of \autoref{thm:lwb2}, for $\eta\in(0,1)$ we introduce the event
\[
A_\eta =A_{\eta, n}= \left\{
	\sup_{y\in D} \frac{\abs{r^{n,t}(y)}}{\sqrt n} \leq \frac\eta2,\ 
	\sup_{y\in D} \frac{\abs{s^{n,t}(y)}}{\sqrt n} \leq \frac\eta2
\right\},
\]
whose probability tends to 1 as $n\to\infty$.
By the contractivity assumption in $W_2$ we have
$W_2^2(\mu^n,\nu^n) \geq e^{2Kt} W_2^2(\mu^{n,t},\nu^{n,t})$, 
therefore it is sufficient to study the asymptotic behaviour of
\[
\Exp\left[W_2^2(\mu^{n,t},\nu^{n,t})\chi_{A_\eta}\right].
\]

To this end, we let $f^{n,t}$ be the solution to \eqref{eq:energy-estimate}, $g^{n,t}$ the solution to the same equation with $s^{n,t}$ in place of $r^{n,t}$ and $h^{n,t}=f^{n,t}-g^{n,t}$. Define $h=-h^{n,t}/\sqrt{n}$, so that $\Delta h = -(r^{n,t}-s^{n,t})/\sqrt n$ and $\norm{\Delta h}_\infty\leq\eta$ in the event $A_\eta$. To this function $h$ we associate the potential $k$ given by \autoref{cor:good-enough-potential}, 
hence we can estimate (in the event $A_\eta$, with $\omega(\eta)$ defined as in \eqref{eq:ambr18} with $f$ replaced by $h$)
\[\begin{split}
\frac12 W_2^2(\mu^{n,t},\nu^{n,t}) &\geq
	\int_D h\d\mu^{n,t} + \int_D k\d\nu^{n,t} \\
&= \int_D (h+k)\d\meas + \int_D h\frac{\d (r^{n,t}-s^{n,t})}{\sqrt n}
	+ \int_D (h+k)\frac{\d s^{n,t}}{\sqrt n} \\
&\geq - e^{\omega(\eta)} \int_D \frac{\abs{\nabla h}^2}2 \d\meas
	+ \int_D \abs{\nabla h}^2 \d\meas + \int_D (h+k)\frac{\d s^{n,t}}{\sqrt n} \\
&\geq \left(1-\frac{e^{\omega(\eta)}}2\right)
	\frac1n \int_D \abs{\nabla h^{n,t}}^2 \d\meas
	- \abs*{\int_D (h+k)\frac{\d s^{n,t}}{\sqrt n}}.
\end{split}\]
Since $h+k\leq0$, we have
\[
\abs*{\int_D (h+k)\frac{\d s^{n,t}}{\sqrt n}} \leq - \eta \int_D (h+k) \d\meas \leq
\eta e^{\omega(\eta)} \int_D \frac{\abs{\nabla h}^2}2 \d\meas,
\]
therefore, still in the event $A_\eta$,
\[
\frac12 W_2^2(\mu^{n,t},\nu^{n,t}) \geq
\left(1-(1+\eta)\frac{e^{\omega(\eta)}}2\right)\frac1n \int_D \abs{\nabla h^{n,t}}^2 \d\meas.
\]

The proof now concludes as before, noticing that, by independence of $\mu^n$ and $\nu^n$,
\[
\begin{split}
\Exp&\left[\int_D \abs{\nabla h^{n,t}}^2\d\meas\right] \\
&= \Exp\left[\int_D \abs{\nabla f^{n,t}}^2\d\meas\right] +
	\Exp\left[\int_D \abs{\nabla g^{n,t}}^2\d\meas\right] +
	2\Exp\left[\int_D \langle\nabla f^{n,t},\nabla g^{n,t}\rangle\d\meas\right] \\
&= 2 \Exp\left[\int_D \abs{\nabla f^{n,t}}^2\d\meas\right] -
	2\Exp\left[\int_D f^{n,t} \Delta g^{n,t}\d\meas\right] \\
&= 2 \Exp\left[\int_D \abs{\nabla f^{n,t}}^2\d\meas\right] -
	2\int_D \Exp[f^{n,t}] \Exp[s^{n,t}] \d\meas
= 2 \Exp\left[\int_D \abs{\nabla f^{n,t}}^2\d\meas\right]
\end{split}
\]
and
\[
\Exp\left[ \int_D \abs{\nabla h^{n,t}}^4 \d\meas \right]
\leq 16 \Exp\left[ \int_D \abs{\nabla f^{n,t}}^4 \d\meas \right]. \qedhere
\]
\end{proof}

From the previous result we get
\[
\liminf_{n\to\infty} \frac{n}{\log n}\Exp\left[W_{2,[0,1]^2}^2(\mu^n,\nu^n)\right]\geq
\liminf_{n\to\infty} \frac{n}{\log n}\Exp\left[W_{2,\T^2}^2(\mu^n,\nu^n)\right]\geq \frac{1}{2\pi}
\]
which, combined with \eqref{eq:ambr21}, concludes the proof \eqref{eq:ambr25}. By looking at \eqref{eq:ambr23} we see also that \eqref{eq:ambr27} holds, and this concludes the proof of \autoref{thm:main}.

\section{A new proof of the AKL lower bound}\label{sec:newAKL}

In this section we see how a minor modification of the ansatz of \cite{CLPS14} provides a new proof of the lower bound in \cite{AKT84}, written in terms of expectations; the upper bound follows immediately from \autoref{thm:main} and H\"older inequality.

The following real analysis lemma is well known, we state it for the case of the flat torus. Its proof (see for instance \cite{AF84}) can be obtained by considering the sublevel sets of the maximal function of $\abs{\nabla h}$.

\begin{lemma}[Lusin approximation of Sobolev functions]\label{lem:lusin}
For all $p>1$, $h\in H^{1,p}(\T^d)$ and all $\lambda>0$ there exists a $\lambda$-Lipschitz function
$\phi:\T^d\to\setR$ with
\begin{equation}\label{eq:ajtai3}
\meas(\{h\neq\phi\})\leq
\frac{C(d,p)}{\lambda^p}\int_{\T^d}\abs{\nabla h}^p\d\meas.
\end{equation}
\end{lemma}

\begin{theorem} If $D=\T^2$ one has
\[
\liminf_{n\to\infty}\sqrt{\frac{n}{\log n}}\Exp\left[W_1(\mu^n,\nu^n)\right]>0.
\]
By the triangle inequality, the same holds for the matching to the reference measure.
\end{theorem}
\begin{proof} As in the proof of the lower bound for $p=2$ we can use contractivity, reducing ourselves to estimating from below the Wasserstein distance between the regularized measures $\mu^{n,t}=u^{n,t}\meas$, $\nu^{n,t}=v^{n,t}\meas$. Let $M>0$ be fixed and set $c(n)=M\sqrt{n^{-1}\log n}$. Let $t=t(n)=\gamma n^{-1}\log n$ with $\gamma$ sufficiently large and let $h^{n,t}$ be as in the proof of the lower bound in the case $p=2$, so that $h=h^{n,t}/\sqrt{n}$ satisfies
\begin{equation}\label{eq:ajtai2}
\lim_{n\to\infty}\frac{n}{\log n}\Exp \left[ \int_D \abs{\nabla h}^2\d\meas\right]=\frac{1}{2\pi}.
\end{equation}
Denote by $\phi$ the $c(n)$-Lipschitz function provided by \autoref{lem:lusin}. We denote by $E_n$ the set $\{h\neq\phi\}$ and from \eqref{eq:ajtai3} and \eqref{eq:ajtai2} we obtain the estimates
\[
\Exp\left[ \meas(E_n) \right] \leq \frac{C(2,2)}{c(n)^2}\Exp\left[ \int_D\abs{\nabla h}^2\d\meas\right] \leq\frac{2 C(2,2)}{2\pi M^2}
\]
%
%
for $n$ large enough, so that we can use H\"older inequality and \eqref{eq:energy-2d-power4} to get,  for some positive constant $C>0$,
\begin{equation}\label{eq:ajtai5}
\Exp \left[ \int_{E_n}\abs{\nabla h}^2\d\meas\right] \leq \frac{C}{M}\frac{\log n}{n}
\end{equation}
for $n$ large enough. Another application of H\"older's inequality yields
\begin{equation}\label{eq:ajtai7}
\Exp \left[ \int_{E_n}\abs{\nabla h}\d\meas\right] \leq\Exp \left[ \int_{E_n}\abs{\nabla h}^2\d\meas\right]^{1/2} \Exp \left[ \meas(E_n) \right]^{1/2} \le \frac{C}{M^{3/2}} \sqrt{\frac{\log n}{n}}
\end{equation}
again for some positive constant $C>0$ (possibly larger than the one in \eqref{eq:ajtai5}).

From Kantorovich's duality formula we get
\[
c(n) W_1(\mu^{n,t},\nu^{n,t})\geq
\abs*{ \int_D \phi (u^{n,t}-v^{n,t})\d\meas } =
\abs*{ \int_D \langle\nabla h,\nabla\phi\rangle \d\meas },
\]
where we used the PDE $\Delta h=u^{n,t}-v^{n,t}$ solved by $h$.
Now we estimate
\[
W_1(\mu^{n,t},\nu^{n,t}) \geq
\frac{1}{c(n)}\int_D \abs{\nabla h}^2\d\meas
	-\frac{1}{c(n)}\abs*{
	\int_{E_n}\langle\nabla h,\nabla h-\nabla \phi\rangle\d\meas }.
\]
By \eqref{eq:ajtai2}, the first term is asymptotic to $(2\pi M)^{-1} \sqrt{n^{-1}\log n}$. We will see that, for $M$ sufficiently large, the first term dominates the second one. Indeed, we have
\[
\frac{1}{c(n)}\abs*{\int_{E_n}\langle\nabla h,\nabla h-\nabla \phi\rangle\d\meas}   \le \frac{1}{c(n)} \left[ \int_{E_n}\abs{\nabla h}^2 + \abs{\nabla h} c(n) \d\meas\right]\]
Taking expectation, using \eqref{eq:ajtai5} and \eqref{eq:ajtai7} we have the inequality, for $n$ sufficiently large,
\[ \frac{1}{c(n)} \Exp \left[ \int_{E_n}\abs{\nabla h}^2 + \abs{\nabla h} c(n) \d\meas\right]  \le C \left(\frac{1}{M^2} + \frac{1}{M^{3/2}} \right)  \sqrt{\frac{\log n}{n}}. \qedhere\]
\end{proof}

\section{Open problems and extensions}\label{sec:more_general_spaces}

In this section we discuss open problems, the present limitations of our technique, and some potential generalizations.

\paragraph{Improvements in the case $p=d=2$.} In this case, the more demanding prediction of \cite{CLPS14} is
\[
\lim_{n\to\infty} \left( \frac{n}{\log n}
	\Exp\left[W_2^2(\mu^n,\nu^n)\right]-\frac 1{2\pi} \right) \log n \in\setR.
\]
This is still open, in this connection notice also that our technique for the lower bound requires $t=\gamma n^{-1}\log n$ with $\gamma$ sufficiently large, while necessarily in the upper bound one is forced to take $t=\gamma n^{-1}\log n$ with $\gamma$ small. Other open problems regard the distribution of the random variables $\frac{n}{\log n} W_2^2(\mu^n,\nu^n)$ and the matching problem involving more general reference measures $\meas$ (the Gaussian case could be interesting, replacing the heat semigroup with the Ornstein-Uhlenbeck semigroup).

\paragraph{Different powers and dimensions.} Our proof in the case $d=2$ exploits the extra room given by the logarithmic correction to the ``natural'' scale $n^{-1/d}$. Let us discuss the difficulties coming from $p\neq 2$ and $d>2$ separately, of course the problem is even more challenging if both things happen.

If $d=2$ and $p=1$, we have already seen in \autoref{sec:newAKL} that the proof can be adapted to obtain the tight lower bound of \cite{AKT84}. Via H\"older's inequality, one obtains the tight upper and lower bounds also for $1<p<2$, and we believe that also the case $p>2$ could be covered, by estimating $\Exp\left[\abs{\nabla f^{n,t}}^k\right]$ with $k$ large integer (we did this for $k=2,4$).
On the other hand, proving convergence of the renormalized expectations seems to require a more precise scheme, since the gradients of solutions to the Monge-Amp\`ere equation describe the optimal transport map $T$ only when $p=2$; in this vein, one could consider (see \cite[Theorem~6.2.4]{AGS08}) the linearizations of
\[
T = \Id - \abs{\nabla\varphi}^{\frac {2-p}{p-1}}\nabla\varphi,\qquad
\rho_1(T)\det(\nabla T)=\rho_0.
\]
In the case $p=1$, an alternative PDE possibility could be given by the construction of the transport density via a $q$-laplacian approximation in \cite{EG99}, $q\to\infty$, which led to the first rigorous proof of the optimal transport map for Monge's problem.

If $p=2$ and $d>2$, the prediction of \cite{CLPS14} is that
\[
n^{2/d}\Exp\left[W_2^2(\mu^n,\nu^n)\right]-c_d =
\frac \xi{2\pi^2}n^{-1+2/d} + o\bigl(n^{-1+2/d}\bigr),
\]
where $c_d$ is not conjectured and the coefficient $\xi$ is explicitly given in terms of the Epstein function. However, our regularization technique seems to fail, even for the purpose of computing $c_d$ (namely proving convergence of the renormalized expectations) or getting tight bounds.
For instance, in the case $d=3$, from \eqref{eq:expansion} we get $\Exp\left[\abs{\nabla f^{n,t}}^2\right]\sim t^{-1/2}$, and therefore one should choose $t\sim  n^{-4/3}$, a regularization time much faster than $n^{-1}$, which does not seem to lead to the density bounds on $\abs{r^{n,t}}/\sqrt{n}$ needed for the proof of the lower bound. On the other hand, the dispersion estimate used in the proof of the upper bound requires $t=o(n^{-2/3})$, a less demanding condition.

\paragraph{A class of abstract metric measure spaces.} We already noticed that in our proof the geometry of the domain enters only through the properties of the heat semigroup $P_t$ with homogeneous Neumann boundary conditions. As a matter of fact, let us briefly indicate how our proof works, still in the case $N=2$, for the class $RCD^*(K,N)$ of ``Riemannian'' metric measure spaces $(X,\dist,\meas)$, extensively studied and characterized in \cite{AGS15}, \cite{AMS15}, \cite{EKS15}.
This class of possibly nonsmooth metric measure spaces, includes for instance all compact Riemannian manifolds without boundary, or ``convex'' manifolds with boundary, namely manifolds having the property that geodesics between any two points do not touch the boundary (as it happens for compact convex domains in $\setR^d$). The class $RCD^*(K,N)$ can be  characterized either in terms of suitable $K$-convexity properties w.r.t.\ $W_2$-geodesics (of the logarithmic entropy for $N=\infty$ \cite{AGS15}, of power entropy \cite{EKS15} or nonlinear diffusion semigroups \cite{AMS15} in the case $N<\infty$), or in terms of Bochner's inequality, very much in the spirit of the Bakry-\'Emery theory (see \cite{BGL14} for a nice introduction to the subject).
In the very recent work \cite{JLZ14}, all regularizing properties of $P_t$ needed for our proof to work have been proved in the context of $RCD^*(K,N)$ spaces. The only missing ingredient in this more abstract framework is the asymptotic expansion of the trace formula provided by \autoref{prop:expansion}, but thanks to \eqref{eq:energy2-n-eps} our results can be stated in terms of the limit
\[
\lim_{t\to 0^+} \sum_{\lambda\in\sigma(\Delta)\setminus\{0\}}
	\frac{e^{2\lambda t}}{\lambda\log t}
\]
whenever it exists.


\end{document}